\documentclass[12pt]{amsart}

\usepackage{stmaryrd}

\usepackage{cite}

\usepackage{tikz}
\usetikzlibrary{decorations.pathreplacing}
\usepackage{float}

\usepackage{amsmath,amssymb}
\usepackage{tikz}
\usetikzlibrary{patterns}

\usepackage{amsfonts,color}
\usepackage{latexsym,amssymb}
\usepackage{amsmath}
\usepackage[utf8]{inputenc}
\usepackage{enumitem}
\usepackage{comment}

\usepackage{ragged2e}
\usepackage{tikz}
\usepackage{float}
\usepackage{booktabs}
\usepackage{subfig}

\newtheorem{theorem}{Theorem}[section]
\newtheorem{lemma}[theorem]{Lemma}
\newtheorem{proposition}[theorem]{Proposition}

\newtheorem{remark}[theorem]{Remark}

\usepackage[bookmarksnumbered,colorlinks, linkcolor=blue, citecolor=red, bookmarks, breaklinks]{hyperref}

\catcode`@=11 \@addtoreset{equation}{section} \catcode`@=12

\usepackage[left=3cm,right=3cm,top=3cm,bottom=3cm]{geometry}

\newcommand{\dx}{\,\mathrm{d}x}

\begin{document}
		
\title[Sharp Adams inequality for Critical Biharmonic Equations]{
Sharp Adams inequality for Critical Biharmonic Equations with Hardy--Hénon Weights in $\mathbb{R}^4$}

\author[J. C. de Albuquerque]{J. C. de Albuquerque}
\address{Departamento de Matem\'atica, Universidade Federal de Pernambuco \\
	50670-901, Recife--PE, Brazil}
\email{josecarlos.melojunior@ufpe.br}
\thanks{The first author was partially supported by CNPq with grant 301849/2025-0.}

\author[J. Carvalho]{J. Carvalho}
\address{Universidade Federal de Sergipe, Departamento de Matem\'atica, 49100-000 São Cristóvão-SE, Brazil}
\email{jonison@mat.ufs.br}\thanks{The second author was partially supported by CNPq with grant 300853/2025-4.}

\author[L. M. Silva]{L. M. Silva}
\address{Departamento de Matem\'atica, Universidade Federal de Pernambuco \\
	50670-901, Recife--PE, Brazil}
\email{lorena.augusto@gmail.com}

		
		
\begin{abstract} 

We study a class of weighted fourth-order elliptic problems in \(\mathbb{R}^4\) driven by the biharmonic operator and involving Hardy--Hénon type potentials. The nonlinear term has critical exponential growth in the sense of an Adams-type inequality. The weights \(|x|^a\) and \(|x|^b\) allow us to incorporate singular, vanishing, and unbounded effects in a unified variational framework. We introduce a suitable weighted function space, establish Sobolev-type embeddings, and prove an Adams-type inequality adapted to this setting. Under appropriate structural assumptions on the nonlinearity, we prove the existence of nontrivial weak solutions. We also obtain regularity properties and derive a Pohozaev identity, which leads to a nonexistence result.
\end{abstract}
		
\maketitle


  \tableofcontents


\section{Introduction}

In this work, we study the following class of fourth-order equations:
\begin{equation}\tag{$\mathcal{P}_\lambda$}\label{P}
	\Delta^2u +|x|^{a}u = \lambda |x|^{b}f(u), \quad \mbox{in}  \  \mathbb{R}^4,
\end{equation}
where $\Delta^{2}u:=\Delta(\Delta u)$ is the biharmonic operator, $a,b\in\mathbb{R}$, $\lambda\in\mathbb{R}$, and $f$ has critical exponential growth in the sense of an Adams-type inequality. More precisely, when $b>-4$, we say that $f$ has critical exponential growth if, for every $\alpha>0$, 
\begin{equation}\label{critical}\tag{$H$}
	\lim_{|s|\to\infty}\frac{|f(s)|}{e^{\alpha s^{2}}}=\left\{
	\begin{array}{ll}
		0,  & \mbox{if } \alpha>32\pi^2(1+b/4) ,
		\\
		\infty,  &\mbox{if } \alpha<32\pi^2(1+b/4).
	\end{array}
	\right.
\end{equation}
Our aim is to provide a comprehensive treatment of the family of problems \eqref{P} by developing the fundamental tools required for a variational analysis in the presence of Hardy- or Hénon-type weights and nonlinearities with critical exponential growth. In particular, we study the existence of weak solutions and some of their qualitative properties. In addition, we establish a Pohozaev identity to analyze situations in which \eqref{P} does not admit nontrivial weak solutions. Before presenting our main results in detail, we provide a brief review of fourth-order equations and some related variants that motivate this work.

\subsection{Overview} Fourth-order equations driven by the biharmonic operator arise naturally in a wide variety of problems in physics, engineering, and applied sciences. They appear, for instance, in models for traveling waves in suspension bridges, in the theory of elastic plates and beams, and in the analysis of mechanical structures with large flexible components, see \cite{Ahmed e Harbi,Lazer e McKenna,Mckenna2,bridge}. Biharmonic models have also been introduced in other contexts, including communication satellites, the propagation of intense laser beams, and the self-focusing of whistler waves in plasmas, see \cite{Pausader,Karpman,Karpman-Shagalov} and references therein. 

In view of this wide range of applications, biharmonic equations have been extensively studied in the literature. From a mathematical perspective, many authors have considered fourth-order elliptic equations of the form
\begin{equation}\label{1}
    \Delta^{2}u-\gamma\Delta u+V(x)u=f(x,u), \quad \mbox{in } \mathbb{R}^{N},
\end{equation}
where $N\geq 4$, $\gamma\geq0$, $V$ is a weight function, and $f$ satisfies a critical growth condition. In the case $N\geq 5$, when $f$ exhibits polynomial growth, we refer the reader to \cite{ALVES-DOO-2002,ALVES-DOO-MIYAGAKI-2001,ALVES-DOO-MIYAGAKI-2003,CHABROWSKI-DOO-2002,FENG-SU-2022,NOUSSAIR-SWANSON-YANG-1992} for studies on the critical case and to \cite{ZHANG-COSTA-2014,BHAKTA-2015,CARRIAO-DEMARQUE-MIYAGAKI-2014,KANG-XU-2017,LIU-CHEN-2015} for equations with singular terms.

The case $N=4$ is a borderline case in the sense of Sobolev embedding. Precisely, if $\Omega\subset\mathbb{R}^{N}$ is a bounded domain, then for Sobolev spaces of order $m$, the classical embedding reads $W_0^{m,p}(\Omega)\hookrightarrow L^{\frac{pN}{N-pm}}(\Omega)$, whenever $pm<N.$ Thus, the limiting case corresponds to $p=N/m$. For first-order Sobolev spaces ($m=1$), the first results were obtained independently by Peetre~\cite{Peetre66}, Pohozaev~\cite{POHO65}, Trudinger~\cite{Trudinger67}, Yudovich~\cite{YUDO61}, and the sharp form of the result was later established by Moser~\cite{Moser}. These works laid the foundations for what is now known in the literature as the \textit{Trudinger--Moser inequality}. Similar results for unbounded domains were established by Adachi and Tanaka \cite{Adachi-Tanaka}, Adimurthi and Yang \cite{Adi-Yang}, Cao \cite{CAO92} and Ruf \cite{Ruf} (see also \cite{DOO01,de Souza-do O,LI-RUF-2008}). The analogous result for higher-order Sobolev spaces is known as \textit{Adams' inequality}, see \cite{Adams}. It has been generalized in several directions and we refer the reader to Ozawa \cite{Ozawa}, Ruf and Sani \cite{RufSani13}, Lam and Lu \cite{Lam-Lu,Lam-Lu2013}, Masmoudi and Sani \cite{MASMOUDI-SANI-2014}, Lu and Yang \cite{Lu-Yang}, Tarsi \cite{Tarsi} and references therein.

In view of the above discussion, if $N=4$, \(m=2\), and \(p=2\), then equation \eqref{1} lies in the borderline Sobolev regime and the critical behavior of $f$ is exponential, in the sense of Adams-type inequalities. Inspired by these results, many authors have studied variants of equation \eqref{1} via variational methods. We refer the reader to \cite{CHEN-2020,MASMOUDI-SANI-2014,MIYAGAKI-2021,Sani13} and \cite{Siberio2019,SOUZA-2024,CHEN-LI-LU-ZHANG-2018,YANG-2012} for problems involving singular weights. 

Henceforth, given a measurable weight \(Q:\mathbb{R}^{4}\to[0,\infty)\) that is positive almost everywhere, we use the following notation: for each $1\leq p<\infty$, the weighted Lebesgue space is defined by
$$
L^p(\mathbb{R}^4,Q(x)):=\left\lbrace u:\mathbb{R}^4 \to \mathbb{R} \ \mbox{measurable}  : \int_{\mathbb{R}^4}Q(x)|u|^p\dx<\infty \right\rbrace, 
$$
equipped with the norm
$$
\|u\|_{L^p(\mathbb{R}^4,Q(x))}:=\left(\int_{\mathbb{R}^4}Q(x)|u|^p\dx\right)^{1/p}.
$$
In \cite{Siberio2019}, motivated by \cite{Su-Wang-Willem1}, the authors studied fourth-order nonlinear Schr\"{o}dinger equations involving radial weights $V,K:(0,\infty)\rightarrow\mathbb{R}$ satisfying the following assumptions: 
\begin{itemize}
\item[($V$)] $V:(0,\infty)\to\mathbb{R}$ is continuous, $V>0$ and there exist $a_0>-4$, $a>-6$ such that
\[
\limsup_{r\to 0^+}\frac{V(r)}{r^{a_0}}<\infty
\quad \text{and} \quad
\liminf_{r\to \infty}\frac{V(r)}{r^{a}}>0;
\]

\item[($K$)] $K:(0,\infty)\to\mathbb{R}$ is continuous, $K>0$ and there exist $b_0>-4$, $b<a$ such that
\[
\limsup_{r\to 0^+}\frac{K(r)}{r^{b_0}}<\infty
\quad \text{and} \quad
\limsup_{r\to \infty}\frac{K(r)}{r^{b}}<\infty.
\]
\end{itemize}
The authors proved the following Adams-type inequality:
\begin{equation}\label{TMWS}
	\mathcal{L}(\alpha,V,K)\left\{
	\begin{array}{ll}
		<\infty,  & \mbox{if } 0<\alpha\leq \min\{32\pi^{2},32\pi^2(1+b_0/4)\},
		\\
		=\infty,  &\mbox{if } \alpha>32\pi^2(1+b_0/4),
	\end{array}
	\right.
\end{equation}
where 
\[
\mathcal{L}(\alpha,V,K):=\sup_{\left\lbrace u \in W_\mathrm{rad} \ : \ \|u\|_{W_\mathrm{rad}} \leq 1\right\rbrace }
\int_{\mathbb{R}^4}K(|x|)\left(e^{\alpha u^{2}}-1\right)\dx
\]
and
$$
W_\mathrm{rad}:=\left\lbrace u\in D^{1,2}_{\mathrm{rad}}(\mathbb{R}^{4})\cap L^{2}(\mathbb{R}^{4}, V(|x|)):|\Delta u|\in L^{2}(\mathbb{R}^{4})\right\rbrace,
$$
endowed with the norm
\begin{equation}\label{norma-Siberio}
    \|u\|_{W_\mathrm{rad}}:=\left[\int_{\mathbb{R}^{4}}\left(|\Delta u|^{2}+|\nabla u|^{2}+V(|x|)u^{2}\right)\,\mathrm{d}x\right]^{1/2}.
\end{equation}
They also proved that $W_\mathrm{rad}$ is continuously and compactly embedded into $L^{q}(\mathbb{R}^{4},K(|x|))$, for $2\leq q <\infty$. As an application, they studied the existence of weak solutions for the following class of nonhomogeneous semilinear elliptic equations:
\begin{equation}\label{eqsiberio}
    \Delta^{2}u-\Delta u+V(|x|)u=K(|x|)f(u)+V(|x|)h(|x|), \quad \mbox{in } \mathbb{R}^{4},
\end{equation}
where $f$ satisfies critical exponential growth and $\|h\|_{L^{2}(\mathbb{R}^{4},V(|x|))}$ is sufficiently small. In this direction, we also cite \cite{SOUZA-2024}, where the authors studied the existence of ``\textit{bound state solutions}'' for the following class of equations 
\begin{equation}\label{eqlorena}
    \Delta^{2}u-\Delta u+V(x)u=K(x)f(x,u), \quad \mbox{in } \mathbb{R}^{4},
\end{equation}
where $V$ and $K$ may decay to zero at infinity and $f$ has critical exponential growth in the sense of an Adams-type inequality.

\subsection{Main results} Our study of equation \eqref{P} is particularly inspired by \cite{Su-Wang-Willem1}. The weights $|x|^{a}$ and $|x|^{b}$ allow one to treat singular, vanishing, and unbounded coefficients in a unified way, interpolating between Hardy-type and H\'enon-type effects. Our aim is twofold. First, we develop the theory associated with an appropriate function space. More precisely, we present a detailed study of weighted Sobolev embeddings, characterize the situations in which such embeddings fail, and prove a weighted Adams-type inequality. Second, we apply these results to the analysis of equation \eqref{P}, investigating the existence, nonexistence, and regularity of solutions.

For $a>-4$, we define the function space
$
E^a:=\overline{C_0^\infty(\mathbb{R}^4)}^{\|\cdot\|_a},
$
where
\begin{equation}\label{nossa-norma}
\|u\|_a:=\left[\int_{\mathbb{R}^4}\left(|\Delta u|^2+|x|^au^2\right)\dx\right]^{1/2}. 
\end{equation}

It is important to note that \(a>-4\) is necessary to ensure that this norm is well defined. Indeed, if $\psi \in C_0^\infty(\mathbb{R}^4)$ satisfies $\psi\equiv1$ in $B_1:=\{x\in\mathbb{R}^4 : |x|<1\}$, then
    \[
    \int_{\mathbb{R}^4}|x|^a\psi^2\dx \geq \int_{B_1}|x|^{a}\dx.
    \]
In this case \(|x|^{a}\in L^{1}(B_1)\) if and only if $a>-4$. Similarly, the condition \(b>-4\) ensures the local integrability of the weight $|x|^{b}$ near the origin, which is required in the definition of \(L^p(\mathbb{R}^4,|x|^b)\).

We also consider the radial space
$$
E^a_{\mathrm{rad}}:=\left\lbrace u \in E^a : g  u = u , \  \forall g \in \mathcal{O}(4)\right\rbrace,
$$
where \(
gu(x):=u(g^{-1}x)
\) and $\mathcal{O}(4)$ denotes the group of orthogonal transformations in $\mathbb{R}^4$. An element in that space can also be denoted by \(u(x)=u(|x|)\).

Our first main result can be stated as follows:

\begin{theorem}[Weighted Sobolev embedding]\label{Teoimersao} Suppose that $a,b>-4$. Then the embedding $E^a_{\mathrm{rad}} \hookrightarrow L^p(\mathbb{R}^4, |x|^b)$ is continuous for $\gamma_{a,b}\leq p< \infty$, where 
	\begin{equation}\label{gamma}
	\gamma_{a,b}:= \max \left\{2, \frac{4(b-a)}{a+4}+2\right\}.   
	\end{equation}
Furthermore, the embedding  $E^a_{\mathrm{rad}} \hookrightarrow L^p(\mathbb{R}^4, |x|^b)$ is compact when $\gamma_{a,b}< p< \infty.$
\end{theorem}
In view of Theorem \ref{Teoimersao}, for each $\nu\geq\gamma_{a,b}$, we denote by \(S_\nu>0\) the best constant of the embedding
$
E^a_{\mathrm{rad}}\hookrightarrow L^\nu(\mathbb{R}^4,|x|^b),
$
that is,
\begin{equation}\label{snu}
S_\nu
:=
\inf_{u\in E^a_{\mathrm{rad}}\setminus\{0\}}
\frac{\|u\|_a^2}
{\left(\displaystyle\int_{\mathbb{R}^4}|x|^b|u|^\nu\dx\right)^{2/\nu}}.
\end{equation}

Now, in order to study \eqref{P} when $f$ has critical exponential growth, we shall prove an Adams-type inequality. We define 
$$
\Phi_\alpha(s):= e^{\alpha s^2}-\sum_{j=0}^{j_0-1}\frac{\alpha^j}{j!}s^{2j}, \quad s \in \mathbb{R},
$$
where $j_0:=\displaystyle\min\{j\in \mathbb{N}:j \geq \gamma_{a,b}/2\}$. Inspired by \cite{Adams,Siberio2019,RufSani13,Sani13}, we obtain the following weighted Adams-type inequality:

\begin{theorem}[Weighted Adams-type inequality]\label{TM}
	Assume that $a,b>-4$. Then, for every $u \in E^a_\mathrm{rad}$ and $\alpha>0$, it holds that $\Phi_\alpha(u) \in L^1(\mathbb{R}^4,|x|^{b})$. Furthermore, 
	$$
	\mathcal{S}(\alpha,a,b):=\sup_{\left\lbrace u \in E^a_\mathrm{rad} \ : \ \|u\|_a \leq 1\right\rbrace }\int_{\mathbb{R}^4}|x|^{b}\Phi_\alpha(u)\dx<\infty,
	$$
	whenever $0<\alpha<32\pi^2(1+b/4)$. Moreover, \(\mathcal{S}(\alpha,a,b)=\infty\) when $\alpha>32\pi^2(1+b/4)$.
\end{theorem}

We also obtain information on the critical case by considering 
    \[
    \mathcal{D}:= \lbrace u\in E^a_{\mathrm{rad}} : 0\leq u(|x|) \leq u(|y|), \ \ \forall \, |y|\leq |x|\rbrace. 
    \]

\begin{proposition}[Critical case]\label{caso-critico}
    If \(\alpha = 32\pi^2(1+b/4)\), then
    \[
    \sup_{\left\lbrace u \in \mathcal{D} \ : \ \|u\|_a \leq 1\right\rbrace } \int_{\mathbb{R}^4}|x|^{b}\Phi_\alpha(u)\dx<\infty.
    \]
\end{proposition}

As an application of the theory developed, we study the existence of solutions to Problem \eqref{P} when $f$ is continuous and satisfies the critical exponential growth \eqref{critical}. For this purpose, we consider the following assumptions:

\begin{enumerate}[label=($f_{1}$),ref=$(f_{1})$]
	\item \label{f1} $\displaystyle\lim_{s\to0}f(s)/|s|^{\gamma_{a,b}-1}=0$, where \(\gamma_{a,b}\) is defined in \eqref{gamma};
\end{enumerate}

\begin{enumerate}[label=($f_{2}$),ref=$(f_{2})$]
	\item \label{f2} there exists $\theta>\gamma_{a,b}$ such that
	$$
	0<\theta F(s):=\theta\int_{0}^{s}f(t)\,dt\leq f(s)s, \quad \mbox{for all} \ s\neq 0;
	$$
\end{enumerate}

\begin{enumerate}[label=($f_{3}$),ref=$(f_{3})$]
	\item \label{f3} there exist \(\mu>0\) and $\nu>\gamma_{a,b}$ such that \(F(s) \geq\mu|s|^\nu\), for all \(s\in \mathbb{R}\).
\end{enumerate}

\noindent An example of nonlinearity satisfying \eqref{critical}, \ref{f1}, \ref{f2} and \ref{f3} is $F(s)=\vert s\vert^{\nu}e^{\alpha_0s^{2}}$, where $\alpha_0=32\pi^2(1+b/4)$ and $\nu>\gamma_{a,b}$.

In this context, we say that $u\in E^a$ is a weak solution to Problem \eqref{P} if it satisfies
\begin{equation}\label{weak}
	\int_{\mathbb{R}^4}\Delta u\Delta v\,\mathrm{d}x+\int_{\mathbb{R}^{4}}|x|^{a}uv\,\mathrm{d}x=\lambda \int_{\mathbb{R}^4}|x|^{b}f(u)v\,\mathrm{d}x, \quad \mbox{for all } v\in C_0^\infty(\mathbb{R}^4).
\end{equation}

Our main result concerning Problem \eqref{P} is stated as follows.

\begin{theorem}\label{T1}
	Suppose that $a,b>-4$, \eqref{critical}, \ref{f1}, \ref{f2} and \ref{f3} hold. Then the following assertions hold:
     \begin{itemize}\vspace{0,2cm}
        \item[(i)] (Existence) If 
        \[
\lambda >\lambda_0:= \max \left\lbrace \frac{S_\nu}{2\mu} , \frac{S_\nu^{\nu/2}}{\mu\nu}\left(\frac{\theta(\nu-2)}{\nu({\theta-2})}\right)^{(\nu-2)/2} \right\rbrace,
\]
where $S_\nu$ is defined in \eqref{snu}, then Problem \eqref{P} admits a nontrivial weak solution.
\vspace{0,1cm}
         \item[(ii)] (Regularity) If $u\in E^a_{\mathrm{rad}}$ is a weak solution of \eqref{P}, then \vspace{0,1cm}
    \begin{itemize}
    \item [$(a)$] there exists \(1<\widetilde{p}<\infty\) such that \(u \in W^{4,\widetilde{p}}_{\mathrm{loc}} (\mathbb{R}^4) \cap C^{0,\widetilde{\sigma}}_{\mathrm{loc}}(\mathbb{R}^4) \cap L^\infty(\mathbb{R}^4)\), for some $\widetilde{\sigma}\in(0,1)$;
\vspace{0,1cm}
    \item[$(b)$] if \(a,b>-1\), then \(u \in C^{3,\sigma}_{\mathrm{loc}}(\mathbb{R}^4)\), for some \(\sigma\in(0,1)\);
\vspace{0,1cm}
    \item[$(c)$] if \(a,b\geq0\) and \(f\) is locally Lipschitz continuous, then \(u \in C^{4,\bar{\sigma}}_{\mathrm{loc}}(\mathbb{R}^4)\), for some \(\bar{\sigma}\in(0,1)\).
\end{itemize}
\vspace{0,1cm}
         \item[(iii)] (Pohozaev identity) If $u$ is a weak solution of \eqref{P}, then it satisfies the Pohozaev identity
\begin{equation}\label{identidade-pohozaev}
\frac{a+4}{2} \int_{\mathbb{R}^4} |x|^a u^2 \, \mathrm{d}x
= \lambda(b+4) \int_{\mathbb{R}^4} |x|^b F(u)\, \mathrm{d}x.       
\end{equation} 
\vspace{0,1cm}
        \item[(iv)] (Nonexistence) If \(a,b>-4\) and \(\lambda\leq0\), then Problem \eqref{P} admits no nontrivial weak solution.
     \end{itemize}
\end{theorem}

\begin{remark}
Our main results extend and complement \cite{Siberio2019} in several directions. The items below make it clear that removing the gradient term from equations \eqref{eqsiberio} and \eqref{eqlorena} makes the problem more delicate, particularly regarding the tools required to treat \eqref{P} variationally.

\vspace{0,2cm}

\begin{itemize}
    \item \textbf{(Radial Lemma):} 
    First, it is worth emphasizing that, by comparing the norms \eqref{norma-Siberio} and \eqref{nossa-norma}, one observes that \(W_{\mathrm{rad}} \hookrightarrow E^a_{\mathrm{rad}}\). In this case, we may use our Lemma \ref{lema-radial} to provide a polynomial decay estimate for functions in \(W_{\mathrm{rad}}\). On the other hand, recalling that in \cite[Lemma 2.1]{Siberio2019} the authors proved that there exist \(C,R_0>0\) such that
\[
|u(x)| \leq \frac{C}{|x|^{(a+6)/4}}\|u\|_{W_{\mathrm{rad}}}, \quad \forall \, u\in W_{\mathrm{rad}} \quad \mbox{and} \quad \forall \, |x|\geq R_0,
\]
it is possible to observe that this inequality does not imply any decay for the elements of our space \(E^a_{\mathrm{rad}}\).

    \vspace{0,2cm}

    \item \textbf{(Embedding result):} In the paper \cite{Siberio2019}, the relation between the exponents \(a\) and \(b\) is given by \(b<a\). In the present work, we also consider the case \(a\leq b\). It is worth mentioning that, due to the embedding \(W_{\mathrm{rad}} \hookrightarrow E^a_{\mathrm{rad}}\), the embedding results obtained in Theorem \ref{Teoimersao} also hold in the space \(W_{\mathrm{rad}}\).

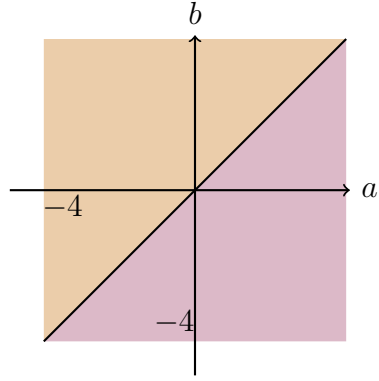
\begin{figure}[H]
\centering
\begin{tikzpicture}[scale=0.5]

\definecolor{rosa}{RGB}{220,185,200}
\definecolor{bege}{RGB}{235,205,170}

\fill[rosa] (-4,-4) -- (4,-4) -- (4,4) -- cycle;
\fill[bege] (-4,-4) -- (-4,4) -- (4,4) -- cycle;

\draw[thick] (-4,-4) -- (4,4);

\draw[->,thick] (-4.9,0) -- (4.1,0) node[right] {$a$};
\draw[->,thick] (0,-4.9) -- (0,4.1) node[above] {$b$};

\node at (-3.5,-0.45) {$-4$};
\node at (-0.55,-3.5) {$-4$};

\node[anchor=north] at (0,-5.2) {
\begin{tabular}{ll}
\raisebox{0.2ex}{\tikz\fill[rosa] (0,0) rectangle (0.3,0.3);}
&
$ \mbox{Literature}: \;  \; b< a:\; 2\le p<\infty$
\\[2mm]
\raisebox{0.2ex}{\tikz\fill[bege] (0,0) rectangle (0.3,0.3);}
&
$\mbox{Our contribution}: \;   \; b\geq a:\; \dfrac{4(b-a)}{a+4}+2\le p<\infty$
\end{tabular}
};
\end{tikzpicture}
\end{figure}
 When \(b<a\), one has \(\gamma_{a,b}=2\). In this case
$
E^a_{\mathrm{rad}}\hookrightarrow L^p(\mathbb{R}^4,|x|^b)
$
is continuous and compact for all \(p\ge2\). The compactness at the endpoint \(p=2\) follows from the additional decay produced by the condition \(b<a\). This recovers, in the present setting, the weighted Sobolev embedding obtained in \cite{Siberio2019} for the space \(W_{\mathrm{rad}}\).

    \vspace{0,2cm}

    \item \textbf{(Weighted Adams-type inequality):}
Analyzing \eqref{TMWS}, we may consider the following cases: 

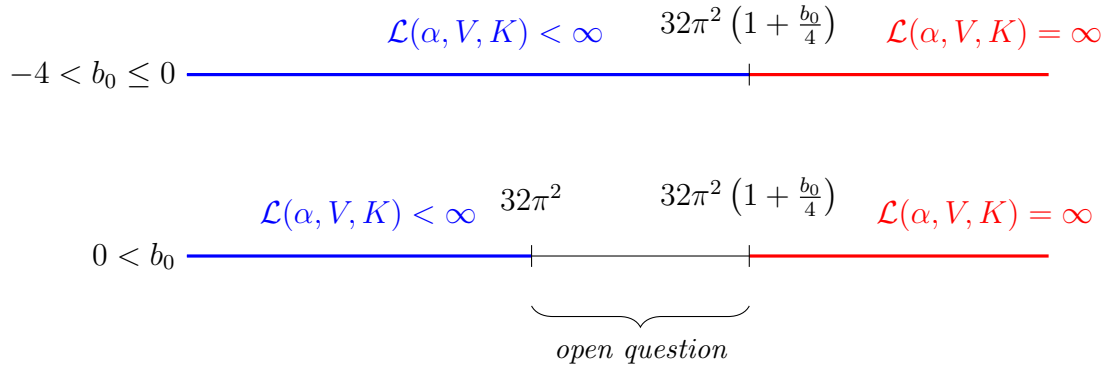
\begin{figure}[H]
\centering
\begin{tikzpicture}[scale=1.2]

\draw[-] (-5,0) -- (4,0);

\draw[blue, very thick] (-5,0) -- (1.2,0);

\draw[red, very thick] (1.2,0) -- (4.5,0);

\draw (1.2,0.12) -- (1.2,-0.12);

\node[left] at (-5,0) {$-4<b_0\leq 0$};

\node[blue] at (-1.6,0.45)
{$\mathcal{L}(\alpha,V,K)<\infty$};

\node[red] at (3.9,0.45)
{$\mathcal{L}(\alpha,V,K)=\infty$};

\node at (1.2,0.55)
{$32{\pi^2}\left(1+\frac{b_0}{4}\right)$};

\draw[-] (-5,-2) -- (4,-2);

\draw[blue, very thick] (-5,-2) -- (-1.2,-2);

\draw[red, very thick] (1.2,-2) -- (4.5,-2);

\draw (-1.2,-1.88) -- (-1.2,-2.12);
\draw (1.2,-1.88) -- (1.2,-2.12);

\node[left] at (-5,-2) {$0< b_0$};

\node[blue] at (-3.0,-1.55)
{$\mathcal{L}(\alpha,V,K)<\infty$};

\node[red] at (3.8,-1.55)
{$\mathcal{L}(\alpha,V,K)=\infty$};

\node at (-1.2,-1.35)
{$32\pi^2$};

\node at (1.2,-1.35)
{$32\pi^2\left(1+\frac{b_0}{4}\right)$};

\draw[decorate,decoration={brace,mirror,amplitude=8pt},black]
(-1.2,-2.55) -- (1.2,-2.55);

\node[black] at (0,-3.05) {\small open question};

\end{tikzpicture}
\end{figure}

In this sense, our contribution is to answer this open question. Indeed, when \(V(|x|)=|x|^a\) and \(K(|x|)=|x|^b\), we have \(\|u\|_a \leq \|u\|_{W_{\mathrm{rad}}}\) for every \(u \in W_{\mathrm{rad}}\). Hence
\begin{equation}\label{comparacao-des-adams}
    \mathcal{L}(\alpha,V,K) \leq \mathcal{S}(\alpha,a,b).
\end{equation}
Hence, \(\mathcal{L}(\alpha,V,K)<\infty\) whenever \(0<\alpha<32\pi^2(1+b/4)\) (see Theorem \ref{TM}) and, in particular, when \(b=b_0\geq0\).
\end{itemize}
\end{remark}

\begin{remark}
    We also note that our approach can be applied to the following class of weighted biharmonic equations
     \begin{equation*}
	\Delta^2u +V(|x|)u =   \lambda K(|x|)f(u), \quad \mbox{in } \mathbb{R}^4,
\end{equation*}
where $V$ and $K$  satisfy assumptions $(V)$ and $(K)$ in the spirit of \cite{Su-Wang-Willem1}.
\end{remark}

\subsection{Outline} The paper is organized as follows. In Section \ref{s2}, we establish a radial lemma and derive the weighted Sobolev embedding results. Furthermore, we prove a weighted Adams-type inequality adapted to the variational framework of the problem. Section \ref{s4} is devoted to equation \eqref{P}, where we investigate the existence and regularity of weak solutions. Moreover, we establish a Pohozaev identity associated with problem \eqref{P} and study cases in which \eqref{P} does not admit nontrivial weak solutions. 

\section{Weighted embeddings and Adams-type inequality}\label{s2}

In this Section, we develop the basic theory for the space $E^a_{\mathrm{rad}}$, by establishing continuous and compact embeddings into weighted Lebesgue spaces, as well as a weighted Adams-type inequality. For this purpose, inspired by \cite{jianfu}, we prove the following Radial Lemma:

\begin{lemma}\label{lema-radial}
	Assume that $a>-4$. Then, for every $u \in E^a_{\mathrm{rad}}$,
	\begin{equation*}\label{eq2}
	 |u(|x|)| \leq  \frac{1}{\sqrt{\omega_3}(a+4)^{1/4}|x|^{(a+4)/4}}\left[\int_{\mathbb{R}^4}\left(|\Delta u|^2+|x|^au^2\right)\dx\right]^{1/2}, \quad x \neq 0,	
	\end{equation*} 
	 where \(\omega_3\) is the area of the sphere in \(\mathbb{R}^4\). 
\end{lemma}

\begin{proof}
	By a standard density argument, it suffices to prove for functions $u \in C^{\infty}_{0,\mathrm{rad}}(\mathbb{R}^4).$ Using $r^3\Delta u(r) = (r^3 u'(r))'$ and H\"older's inequality, we have that
$$
	r^3 |u'(r)| \leq \int_{0}^{r}|\Delta u(s)|s^3 \, \mathrm{d}s \leq \left(\int_{0}^{r} s^3 \, \mathrm{d}s\right)^{1/2}\left(\int_{0}^{r}  |\Delta u(s)|^2s^3\, \mathrm{d}s\right)^{1/2}.
	$$
As a consequence,
	$$
	r^3 |u'(r)| \leq \frac{1}{2\sqrt{\omega_3}} r^{2}\|\Delta u\|_{L^2(\mathbb{R}^4)}.
	$$
Thus,
    \begin{equation}\label{decaimento-derivada}
		 |u'(r)| \leq \frac{1}{2\sqrt{\omega_3}}  r^{-1}\|\Delta u\|_{L^2(\mathbb{R}^4)}, \quad  \forall \, r>0.
	\end{equation}
This inequality implies that
\[
u^2(r)
=
-2\int_r^{\infty}u(s)u'(s)\,\mathrm{d}s
\leq
\frac{1}{\sqrt{\omega_3}}
\|\Delta u\|_{L^2(\mathbb{R}^4)}
\int_r^{\infty}|u(s)|\,s^{-1}\,\mathrm{d}s.
\]
Using Hölder's inequality, we get
\[
\begin{aligned}
\int_r^{\infty}|u(s)|\,s^{-1}\,\mathrm{d}s
&\leq
\left(
\int_r^{\infty}u^2(s)s^{a+3}\,\mathrm{d}s
\right)^{1/2}
\left(
\int_r^{\infty}s^{-a-5}\,\mathrm{d}s
\right)^{1/2}  
\\
& \leq
\frac{r^{-(a+4)/2}}
{\sqrt{\omega_3(a+4)}}
\|u\|_{L^2(\mathbb{R}^4,|x|^a)}.
\end{aligned}
\]
These estimates infer that
\[
u^2(r) \leq \frac{r^{-(a+4)/2}}{\omega_3\sqrt{a+4}}\|u\|_a^2.
\]
Therefore,
$$
|u(r)| \leq \frac{1}{\sqrt{\omega_3}(a+4)^{1/4}}\frac{\|u\|_a}{r^{(a+4)/4}},
$$
which finishes the proof.
\end{proof}

We recall the Gagliardo-Nirenberg interpolation inequality, see \cite{Gagliardo,Nirenberg}.

\begin{theorem}\label{GNS} Let $1\leq k_1,k_2,k_3\leq \infty$, let $N,j,m \in \mathbb{N}$ with $j<m$, and let $\theta \in [0,1]$ satisfy
	$$
	\frac{1}{k_1} = \frac{j}{N} + \left(\frac{1}{k_3} - \frac{m}{N}\right)\theta + \frac{1-\theta}{k_2} \quad \mbox{and} \quad \frac{j}{m} \leq \theta \leq 1.
	$$
	Then, there exists $C>0$, independent of $u$, such that
	$$
	\|D^ju\|_{L^{k_1}(\mathbb{R}^N)} \leq C\|D^m u\|_{L^{k_3}(\mathbb{R}^N)}^\theta\|u\|_{L^{k_2}(\mathbb{R}^N)}^{1-\theta}, \quad \forall \, u \in L^{k_2}(\mathbb{R}^N) \cap W^{m,k_3}(\mathbb{R}^N).
	$$
\end{theorem}

By choosing $k_3=2, j=1, m=2$, and $\theta=1$ in Theorem~\ref{GNS}, we obtain
$$
\frac{1}{k_1} = \frac{1}{N} + \frac{1}{2} - \frac{2}{N} = -\frac{1}{N} + \frac{1}{2} = \frac{N-2}{2N},
$$
which implies that $k_1 = 2^*:=2N/(N-2)$. For $u \in C_0^\infty(\mathbb{R}^N)$, integration by parts gives
\[
\int_{\mathbb{R}^N}(u_{x_ix_j})^2\, \mathrm{d}x =- \int_{\mathbb{R}^N}u_{x_i}(u_{x_ix_j})_{x_j}\, \mathrm{d}x = - \int_{\mathbb{R}^N}u_{x_i}(u_{x_jx_j})_{x_i}\, \mathrm{d}x = \int_{\mathbb{R}^N}u_{x_ix_i}u_{x_jx_j}\,\mathrm{d}x.
\]
Thus, we deduce
\begin{equation}\label{laplaciano-derivada-D2}
\begin{aligned}
\int_{\mathbb{R}^N}|D^2 u|^2\, \mathrm{d}x &= \int_{\mathbb{R}^N}\sum_{i,j=1}^N(u_{x_ix_j})^2\, \mathrm{d}x=\int_{\mathbb{R}^N}\left(\sum_{i=1}^N u_{x_ix_i}\right)\left(\sum_{j=1}^N u_{x_jx_j}\right)\, \mathrm{d}x 
\\
&= \int_{\mathbb{R}^N}|\Delta u|^2\, \mathrm{d}x.    
\end{aligned}    
\end{equation}
Therefore,
\begin{equation}\label{des-GNS-2**}
	\|\nabla u\|_{L^{2^*}(\mathbb{R}^N)} \leq C \|D^2 u\|_{L^2(\mathbb{R}^N)}= C \|\Delta u\|_{L^{2}(\mathbb{R}^N)}, \quad \forall \, u \in C_0^\infty(\mathbb{R}^N). 
\end{equation}

We are now ready to prove the following result.


\begin{lemma}\label{lema-imersao-local}
	Assume that $a>-4$. Then, for each $R>0$, the embedding $E^a \hookrightarrow H^2 (B_R)$ is continuous. In particular, $E^a \hookrightarrow L^t (B_R)$ continuously for every \(1\leq t<\infty\).
\end{lemma}

\begin{proof}

Fix \(u \in E^a\) and \(R>0\). We divide the proof into three steps:

\bigskip

\noindent\underline{Step 1:} \(u\in L^2(B_R)\) and there exists \(C_1>0\) such that
    \[
    \int_{B_R}u^2\dx \leq C_1\|u\|_a^2.
    \]


In fact, by definition of \(E^a\), we can consider $(u_{n})_n\subset C^{\infty}_{0}(\mathbb{R}^{4})$ satisfying $\|u_{n}-u\|_{a}\rightarrow0$, as $n\rightarrow\infty$. Next, let $\phi\in C^{\infty}_{0}(\mathbb{R}^{4})$ be such that $\mathrm{supp}(\phi)\subset B_{R+1}$, $\phi\equiv 1$ in $B_{R}$, and $0\leq \phi\leq 1$. Hence, $(u_{n}-u_{m})\phi\in H^{1}_{0}(B_{R+1})$. In view of Poincar\'e inequality,
$$
	\int_{B_R}|u_{n}-u_{m}|^{2}\,\mathrm{d}x  \leq \int_{B_{R+1}}|(u_{n}-u_{m})\phi|^{2}\,\mathrm{d}x \leq C_2\int_{B_{R+1}}|\nabla\left((u_{n}-u_{m})\phi\right)|^{2}\,\mathrm{d}x.
	$$
As a consequence,
	$$
		\int_{B_R}|u_{n}-u_{m}|^{2}\,\mathrm{d}x   \leq C_{3} \int_{B_{R+1}}|\nabla u_{n}-\nabla u_{m}|^{2}\phi^{2}\,\mathrm{d}x + C_{3} \int_{B_{R+1}\backslash B_{R}}|u_{n}-u_{m}|^{2}|\nabla \phi|^{2}\,\mathrm{d}x.
	$$
   Since $0\leq \phi\leq 1$, we can conclude that
   \begin{equation}\label{des-u_n-L2}
    \begin{aligned}
		\int_{B_R}|u_{n}-u_{m}|^{2}\,\mathrm{d}x
		&\leq C_{3} \int_{B_{R+1}}|\nabla u_{n}-\nabla u_{m}|^{2}\,\mathrm{d}x \\ 
        &\quad+  \frac{C_{3}\|\nabla\phi\|_{L^\infty(\mathbb{R}^4)}^2}{\inf_{B_{R+1}\backslash B_{R}} |x|^{a}} \int_{B_{R+1}\backslash B_{R}}|x|^{a}|u_{n}-u_{m}|^{2}\,\mathrm{d}x.
		\end{aligned}   
   \end{equation}
    Moreover, since $(u_{n})_n$ is Cauchy in $E^a$, it follows from \eqref{des-GNS-2**} that
\begin{equation}\label{cauchy}
		\|\nabla u_{n}-\nabla u_{m}\|_{L^{2^*}(\mathbb{R}^4)}
    \leq C \|\Delta u_{n}-\Delta u_{m}\|_{L^2(\mathbb{R}^4)}\rightarrow0, \quad m,n\rightarrow\infty,
	\end{equation}
    where $2^*=4$. Thus, in view of \eqref{cauchy} and H\"older's inequality, we deduce
    \begin{equation}\label{cauchy-derivada-bola}
        \begin{aligned}
		\int_{B_{R+1}}|\nabla u_{n}-\nabla u_{m}|^{2}\,\mathrm{d}x & \leq |B_{R+1}|^{\frac{2^{*}-2}{2^{*}}}\left(\int_{B_{R+1}}|\nabla u_{n}-\nabla u_{m}|^{2^{*}}\,\mathrm{d}x\right)^{2/2^{*}}
        \\
		&\leq |B_{R+1}|^{\frac{2^{*}-2}{2^{*}}}C^2 \|\Delta u_{n}-\Delta u_{m}\|_{L^2(\mathbb{R}^4)}^{2}
        \\
        &\leq |B_{R+1}|^{\frac{2^{*}-2}{2^{*}}}C^2 \|u_{n}- u_{m}\|_{a}^{2}\rightarrow 0,
	\end{aligned}
    \end{equation}
	as $m,n\rightarrow\infty$. Here, $|B_{R+1}|$ denotes the Lebesgue measure of the ball. This estimate, combined with  \eqref{des-u_n-L2}, implies that
\begin{equation}\label{ime}
		\int_{B_R}|u_{n}-u_{m}|^{2}\,\mathrm{d}x \leq C_4\|u_{n}-u_{m}\|_{a}^{2}\rightarrow 0,
	\end{equation}
	as $n,m\rightarrow\infty$. Thus, $(u_{n})_n$ is Cauchy in $L^{2}(B_{R})$, and so \(u_{n}\rightarrow v\) in \(L^{2}(B_{R})\). Hence, up to a subsequence, $u_{n}(x)\rightarrow v(x)$ a.e. in $B_{R}$. On the other hand, since $\|u_{n}-u\|_{a}\rightarrow0$ as $n\rightarrow\infty$, it follows, after passing to a subsequence, that \(|x|^{a/2}(u_n(x)-u(x))\to 0\) a.e. in \(\mathbb{R}^4\). Consequently, $u_{n}(x)\rightarrow u(x)$ a.e. in $\mathbb{R}^4$, implying \(u=v\) in \(B_R\). So, \(u \in L^2(B_R)\). Moreover, in view of \eqref{ime}, we obtain 
	\begin{equation*}\label{ime2}
		\int_{B_{R}}u^{2}\,\mathrm{d}x \leq \liminf_{n\rightarrow\infty}\int_{B_{R}}u_{n}^{2}\,\mathrm{d}x \leq C_5\liminf_{n\rightarrow\infty}\|u_{n}\|_{a}^{2} = C_5\|u\|_{a}^{2},
	\end{equation*}	
    completing this step.

    \bigskip
  	
  	\noindent\underline{Step 2:} There exists \(C_6>0\) such that
    \[
    \int_{B_R}|\nabla u|^2\dx \leq C_6\|u\|_a^2.
    \]


Applying \eqref{cauchy-derivada-bola}, we have that the sequence $\left({\partial u_{n}}/{\partial x_{i}}\right)_n$ is Cauchy in $L^{2}(B_{R})$, for each $i=1,2,3,4$. Hence, ${\partial u_{n}}/{\partial x_{i}}\rightarrow v_{i}$, as $n\rightarrow\infty$, in $L^{2}(B_{R})$, for each $i=1,2,3,4$. For $\psi\in C^{\infty}_{0}(B_R)$ arbitrary, there holds
\begin{equation}\label{derivada-fraca-un}
    \int_{B_R}u_{n}\frac{\partial \psi}{\partial x_{i}}\,\mathrm{d}x=-\int_{B_R}\frac{\partial u_{n}}{\partial x_{i}}\psi\,\mathrm{d}x.
\end{equation}
By H\"{o}lder’s inequality and the convergence \(u_{n}\rightarrow u\) in \(L^{2}(B_{R})\), we have that
\[
\left| \int_{B_R}(u_{n}-u)\frac{\partial \psi}{\partial x_{i}}\,\mathrm{d}x \right| \leq \|u_n-u\|_{L^2(B_R)}\left\|\frac{\partial \psi}{\partial x_{i}}\right\|_{L^2(B_R)}\to 0,
\]
as \(n\to\infty\). Similarly,
\[
\left| \int_{B_R}\left(\frac{\partial u_{n}}{\partial x_{i}}-v_i\right)\psi\,\mathrm{d}x \right|\to 0,
\]
as \(n\to\infty\). Combining these convergences with \eqref{derivada-fraca-un}, we obtain
\begin{equation}\label{vi-derivada-de-u}
 \int_{B_R}u\frac{\partial \psi}{\partial x_{i}}\,\mathrm{d}x=-\int_{B_R}v_i\psi\,\mathrm{d}x, 
\end{equation}
which implies that ${\partial u}/{\partial x_{i}}=v_{i}$, and consequently $|\nabla u|\in L^{2}(B_{R})$. In view of Fatou's Lemma and \eqref{cauchy-derivada-bola}, we deduce
\begin{equation}\label{ime1}
		\int_{B_{R}}|\nabla u|^{2}\,\mathrm{d}x \leq \liminf_{n\rightarrow\infty}\int_{B_{R}}|\nabla u_{n}|^{2}\,\mathrm{d}x \leq C_{7} \liminf_{n\rightarrow\infty}\|u_{n}\|_{a}^{2} = C_{7}\|u\|_{a}^{2}.
	\end{equation}	
    This concludes the second step.

\bigskip
  	
  	\noindent\underline{Step 3:} It holds that
    \[
    \int_{B_R}|D^2 u|^2\dx \leq \|u\|_a^2.
    \]


Indeed, from \eqref{laplaciano-derivada-D2}, the sequence $\left({\partial^2 u_{n}}/{\partial x_{i}\partial x_{j}}\right)_n$ is Cauchy in $L^{2}(\mathbb{R}^4)$, for each $i,j=1,2,3,4$. Thus, we can conclude that \({\partial^2 u_{n}}/{\partial x_{i}\partial x_{j}} \to w_{i,j}\) in $L^{2}(\mathbb{R}^4)$, as \(n\to \infty\), with $i,j=1,2,3,4$. So, given \(\varphi \in C^{\infty}_{0}(B_R)\), we have that
\[
	\int_{B_R}\frac{\partial^2 u_{n}}{\partial x_{i}\partial x_{j}}\varphi\,\mathrm{d}x=\int_{B_R}u_{n}\frac{\partial^2 \varphi}{\partial x_{i}\partial x_{j}}\,\mathrm{d}x,
	\]
    for any $i,j=1,2,3,4$. Arguing as in the proof of \eqref{vi-derivada-de-u}, we obtain
\[
	\int_{B_R}w_{i,j}\varphi\,\mathrm{d}x=\int_{B_R}u\frac{\partial^2 \varphi}{\partial x_{i}\partial x_{j}}\,\mathrm{d}x,
	\]
    for any $i,j=1,2,3,4$. Therefore, \(w_{i,j} = {\partial^2 u}/{\partial x_{i}\partial x_{j}}\), for any $i,j=1,2,3,4$, and so
\[
    \begin{aligned}
        \int_{B_R}|D^2 u|^2\dx \leq 
        \liminf_{n\to \infty}\int_{\mathbb{R}^4}|D^2 u_n|^2\dx = \lim_{n\to \infty}\int_{\mathbb{R}^4}|\Delta u_n|^2\dx = \int_{\mathbb{R}^4}|\Delta u|^2\dx \leq \|u\|_a^2,
    \end{aligned}
    \]
where we used \eqref{laplaciano-derivada-D2}, finalizing this step.

    Combining the three steps completes the proof.
	\end{proof}

Now, we prove the main result of this section.

\begin{proof}[\textbf{Proof of Theorem \ref{Teoimersao}}]
		For any $R>0$ and $\gamma_{a,b}\leq p< \infty$, we write
		\begin{equation}\label{soma-bola-fora-da-bola}
			\int_{\mathbb{R}^4} |x|^b|u|^p \dx = \int_{B_R}  |x|^b|u|^p \dx + \int_{\mathbb{R}^4  \setminus B_R} |x|^b|u|^p \dx.
		\end{equation}
Since $b> -4,$ we can choose $p_1>1$ close enough to $1$ such that $|x|^{p_1b} \in L^1(B_R)$. Thus, using Hölder's inequality together with $E^a_{\mathrm{rad}} \hookrightarrow H^2(B_R) $ (see Lemma~\ref{lema-imersao-local}) and $H^2(B_R) \hookrightarrow L^{p_2p}(B_R)$, we get
\begin{equation}\label{Qbola}
			\int_{B_R} |x|^b|u|^p \dx \leq C_1 \left( \int_{B_R} |x|^{p_1 b} \dx\right)^{1/p_1} \left( \int_{B_R} |u|^{p_2p} \dx \right)^{1/p_2} \leq C_2 \|u\|^p_a.
		\end{equation}		

Next, we estimate the second integral on the right-hand side of \eqref{soma-bola-fora-da-bola}. Applying Lemma \ref{lema-radial}, one deduces
        $$
		\int_{\mathbb{R}^4  \setminus B_R}  |x|^b|u|^p \dx= \int_{\mathbb{R}^4  \setminus B_R}  |x|^b|u|^{p-2} u^2 \dx\leq C_3 \|u\|_a^{p-2} \int_{\mathbb{R}^4  \setminus B_R}  |x|^{\zeta}|x|^au^{2} \dx, 
		$$
where
		$$
		\zeta:=(b-a)-(p-2)\frac{(a+4)}{4}.
		$$
Since $a>-4$ and $p \geq \gamma_{a,b}$, it follows that $\zeta
\leq 0$. Thus,
\begin{equation}\label{Qcompl}
	\int_{\mathbb{R}^4  \setminus B_R}  |x|^{b}|u|^p \dx \leq C_4 R^{\zeta}\|u\|_a^{p}. 
\end{equation}
Therefore, in view of  \eqref{soma-bola-fora-da-bola}, \eqref{Qbola}, and \eqref{Qcompl}, we can conclude that the embedding  $E^a_{\mathrm{rad}} \hookrightarrow L^p(\mathbb{R}^4, |x|^b)$ is continuous for $\gamma_{a,b}\leq p< \infty$, completing the first part of the proof.

To prove compactness, let \((u_n)_n\) be a sequence in \(E^a_{\mathrm{rad}}\) such that \(u_n \rightharpoonup 0\) in \(E^a_{\mathrm{rad}}\). In light of Rellich-Kondrachov Theorem, it follows from \eqref{Qbola} that
\[
\lim_{n\to\infty}\int_{B_R}|x|^b|u_n|^p\dx =0.
\]
On the other hand, for \(\varepsilon>0\) arbitrary, if \(p>\gamma_{a,b}\), then we can choose \(R>0\) large in \eqref{Qcompl} such that
\[
\int_{\mathbb{R}^4  \setminus B_R}  |x|^{b}|u_n|^p \dx<\varepsilon.
\]
This, combined with the preceding convergence, completes the proof.
\end{proof}


Now, we focus on the proof of the weighted Adams-type inequality.

\begin{proof}[\textbf{Proof of Theorem \ref{TM}}]
	Given $u \in E^a_\mathrm{rad}$ and $\alpha,R>0$, we write
    \begin{equation}\label{soma-exponencial-bola-complementar}
        \int_{\mathbb{R}^4}|x|^{b}\Phi_\alpha(u)\dx = \int_{B_R}|x|^{b}\Phi_\alpha(u)\dx + \int_{\mathbb{R}^4\setminus B_R}|x|^{b}\Phi_\alpha(u)\dx.
    \end{equation}
    By Lemma \ref{lema-imersao-local}, $u \in H^2(B_R)$. Hence,
    \begin{equation}\label{definicao-v}
    v(|x|):=u(|x|)-u(R) \in H^1_0(B_R) \cap H^2(B_R).    
    \end{equation}
    Since
	$$
	u^2 = (v+u(R))^2 \leq (1+\delta)v^2+C(\delta)u^2(R),
	$$
	for every $\delta>0$, one may deduce
    \[\int_{B_R}|x|^{b}\Phi_\alpha(u)\dx \leq \int_{B_R}|x|^{b}e^{\alpha u^2}\dx \leq e^{\alpha C(\delta)u^2(R)}	\int_{B_R}|x|^{b}e^{\alpha (1+\delta) v^2}\dx.
    \]
    Making the change of variables \(y=x/R\), we obtain
    \begin{equation}\label{estimativa-u-w}
        \int_{B_R}|x|^{b}\Phi_\alpha(u)\dx \leq e^{\alpha C(\delta)u^2(R)}R^{b+4}\int_{B_1}|y|^be^{\alpha(1+\delta)w^2}\,\mathrm{d}y,
    \end{equation}
    where \(w(y):=v(Ry) \in H^1_0(B_1) \cap H^2(B_1)\).

    If \(-4<b\leq0\), then we can apply \cite[Theorem 1.2]{Lam-Lu} to get 
\begin{equation*}
        \int_{B_R}|x|^{b}\Phi_\alpha(u)\dx<\infty.
    \end{equation*}
    Moreover, if \(\|u\|_a\leq 1\), then \(\|\Delta w\|_{L^2(B_1)} = \|\Delta v\|_{L^2(B_R)} =  \|\Delta u\|_{L^2(B_R)}\leq 1\). Applying \cite[Theorem 1.2]{Lam-Lu} once more to \eqref{estimativa-u-w}, we conclude that
    \[
    \int_{B_R}|x|^{b}\Phi_\alpha(u)\dx \leq e^{\alpha C(\delta)u^2(R)}R^{b+4}\sup_{\left\lbrace \omega \in H^1_0(B_1) \cap H^2(B_1) \ : \ \|\Delta\omega\|_{L^2(B_1)} \leq 1\right\rbrace }\int_{B_1}\frac{e^{\alpha(1+\delta)\omega^2}}{|y|^{-b}}\,\mathrm{d}y,
    \]
where we take \(\delta>0\) such that \(\alpha(1+\delta)<32\pi^2(1+b/4)\).

On the other hand, if \(b>0\), then we can use \cite[Theorem 1]{Calanchi-Tarsi} together with \eqref{estimativa-u-w}
 to conclude that 
 \[
 \int_{B_R}|x|^{b}\Phi_\alpha(u)\dx<\infty
 \]
 and   
\[\int_{B_R}|x|^{b}\Phi_\alpha(u)\dx \leq e^{\alpha C(\delta)u^2(R)}R^{b+4}\sup_{\left\lbrace \omega \in H^1_0(B_1) \cap H^2(B_1) \ : \ \|\Delta\omega\|_{L^2(B_1)} \leq 1\right\rbrace }\int_{B_1}|y|^be^{\alpha(1+\delta)\omega^2}\,\mathrm{d}y,
    \]
    whenever \(\|u\|_a \leq 1\) and $0<\alpha<32\pi^2(1+b/4)$. Moreover, by Lemma \ref{lema-radial}, for every \(u\in E_{\mathrm{rad}}^{a}\) with
\(\|u\|_{a}\leq 1\), we have
\[
|u(R)|
\leq C_{a}R^{-(a+4)/4}\|u\|_{a}
\leq C_{a}R^{-(a+4)/4}.
\] 
Consequently,
\[
\int_{B_R}|x|^{b}\Phi_\alpha(u)\dx \leq C_1,
\]
with \(C_1>0\) independent of \(u\). In both cases, we have that     \begin{equation}\label{integravel-bola}
       \int_{B_R}|x|^{b}\Phi_\alpha(u)\dx<\infty
    \end{equation}
     and
    \begin{equation}\label{TM-bola}
        \sup_{\left\lbrace u \in E^a_\mathrm{rad} \ : \ \|u\|_a \leq 1\right\rbrace }\int_{B_R}|x|^{b}\Phi_\alpha(u)\dx<\infty, \quad \mbox{if} \ \  0<\alpha<32\pi^2(1+b/4).
    \end{equation}

Next, we estimate the integral over \(\mathbb{R}^4 \setminus B_R\). To do this, we first write 
\[
\begin{aligned}
    \int_{\mathbb{R}^4\setminus B_R}|x|^{b}\Phi_\alpha(u)\dx = \sum_{j=j_0}^\infty \frac{\alpha^j}{j!} \int_{\mathbb{R}^4\setminus B_R}|x|^{b}|u|^{2j}\dx = \sum_{j=j_0}^\infty \frac{\alpha^j}{j!} \int_{\mathbb{R}^4\setminus B_R}|x|^{b}|u|^{2j-\gamma_{a,b}} |u|^{\gamma_{a,b}}\dx.
\end{aligned}
\]
Lemma \ref{lema-radial} yields
\[
|u|^{2j-\gamma_{a,b}} \leq  \frac{C_2^{2j-\gamma_{a,b}}}{|x|^{(a+4)(2j-\gamma_{a,b})/4}}\|u\|_a^{2j-\gamma_{a,b}} , \quad x \neq 0.
\]
Substituting this estimate into the preceding identity and applying Theorem \ref{Teoimersao}, we obtain 
\[
\begin{aligned}
    \int_{\mathbb{R}^4\setminus B_R}|x|^{b}\Phi_\alpha(u)\dx &\leq \left(\frac{C_2\|u\|_a}{R^{(a+4)/4}}\right)^{-\gamma_{a,b}}\sum_{j=j_0}^\infty \left[\frac{\alpha C_2^2\|u\|_a^2}{R^{(a+4)/2}}\right]^j\frac{1}{j!} \int_{\mathbb{R}^4\setminus B_R}|x|^{b}|u|^{\gamma_{a,b}}\dx
    \\
    &\leq C_3 e^{{\alpha C_2^2\|u\|_a^2}/{R^{(a+4)/2}}}<\infty.
\end{aligned}
\]
Applying this inequality and \eqref{integravel-bola} in \eqref{soma-exponencial-bola-complementar}, we obtain $\Phi_\alpha(u) \in L^1(\mathbb{R}^4,|x|^{b})$, for any \(\alpha>0\) and \(u\in E^a_{\mathrm{rad}}\). The preceding estimate also yields
\[
\sup_{\left\lbrace u \in E^a_\mathrm{rad} \ : \ \|u\|_a \leq 1\right\rbrace } \int_{\mathbb{R}^4\setminus B_R}|x|^{b}\Phi_\alpha(u)\dx \leq C_3 e^{{\alpha C_2^2}/{R^{(a+4)/2}}}.
\]
Combining this estimate with \eqref{TM-bola} and \eqref{soma-exponencial-bola-complementar} shows that \(\mathcal{S}(\alpha,a,b)<\infty\), whenever $0<\alpha<32\pi^2(1+b/4)$.

Finally, we show that \(\mathcal{S}(\alpha,a,b)=\infty\) when $\alpha>32\pi^2(1+b/4)$. To do this, let us consider the family of test functions introduced in \cite{Lu-Yang}. Namely, for \(0<\varepsilon<1\), we define
\[
u_{\varepsilon}(x):=
\begin{cases}
\displaystyle
\sqrt{\frac{-\log\varepsilon}{32\pi^{2}}}
-\frac{|x|^{2}}{\sqrt{-8\pi^{2}\varepsilon\log\varepsilon}}
+\frac{1}{\sqrt{-8\pi^{2}\log\varepsilon}},
& \text{if } |x|<\varepsilon^{1/4},
\\[0.4cm]
\displaystyle
\frac{-\log|x|}{\sqrt{-2\pi^{2}\log\varepsilon}},
& \text{if } \varepsilon^{1/4}\leq |x|\leq 1,
\\[0.4cm]
\eta_{\varepsilon},
& \text{if } |x|>1,
\end{cases}
\]
where \(\eta_{\varepsilon}\in C_{0}^{\infty}(B_{2})\) satisfies 
\[
\eta_{\varepsilon}(x)=0,
\quad
\frac{\partial\eta_{\varepsilon}}{\partial\nu}(x)
=
\frac{1}{\sqrt{-2\pi^{2}\log\varepsilon}},
\quad
\forall \, x\in\mathbb{R}^{4},\quad |x|=1,
\]
\[
\|\eta_{\varepsilon}\|_{L^\infty(\mathbb{R}^4)},
\quad
\|\nabla\eta_{\varepsilon}\|_{L^\infty(\mathbb{R}^4)},
\quad\text{and}\quad
\|\Delta\eta_{\varepsilon}\|_{L^\infty(\mathbb{R}^4)}
\]
are all
\[
O\left(\frac{1}{\sqrt{-\log\varepsilon}}\right),
\quad\text{as }\varepsilon\to0^{+}.
\]
Furthermore, in \cite{Lu-Yang} it was proved that \(u_\varepsilon \in H^1_0(B_1) \cap H^2(B_1)\) and 
\begin{equation}\label{laplaciano-L2-zero}
\|\Delta u_{\varepsilon}\|_{L^2(B_1)}^{2}
=
1+O\left(\frac{-1}{\log\varepsilon}\right).
\end{equation}
Now, we can estimate
\[
\begin{aligned}
    \int_{\mathbb{R}^{4}}|x|^au_{\varepsilon}^{2}\dx &\leq C_4(-\log\varepsilon)
\int_{0}^{\varepsilon^{1/4}}s^{a+3}\,\mathrm{d}s + \frac{C_4}{(-\varepsilon \log\varepsilon)}\int_{0}^{\varepsilon^{1/4}}s^{a+7}\,\mathrm{d}s 
\\
&\quad+ \frac{C_4}{(- \log\varepsilon)}\int_{0}^{\varepsilon^{1/4}}s^{a+3}\,\mathrm{d}s +\left(\frac{C_{4}}{-\log\varepsilon}\right)
\int_{\varepsilon^{1/4}}^{1}
s^{a+3}|\log s|^{2}\,\mathrm{d}s
\\
&\quad + C_4\|\eta_{\varepsilon}\|_{L^\infty(\mathbb{R}^4)}^{2}
\int_{1}^{2}s^{a+3}\,\mathrm{d}s.
\end{aligned}
\]
Note that
\[
(-\log\varepsilon)
\int_{0}^{\varepsilon^{1/4}}s^{a+3}\,\mathrm{d}s = C_5(-\log\varepsilon) \varepsilon^{(a+4)/4} \to 0,
\]
as \(\varepsilon\to 0^+\), because \(a>-4\). Similarly, we can conclude that
\[
\frac{1}{(-\varepsilon \log\varepsilon)}\int_{0}^{\varepsilon^{1/4}}s^{a+7}\,\mathrm{d}s \to 0 \quad \mbox{and} \quad \frac{1}{(- \log\varepsilon)}\int_{0}^{\varepsilon^{1/4}}s^{a+3}\,\mathrm{d}s \to 0,
\]
as \(\varepsilon\to 0^+\). On the other hand, using that 
\[
\int_0^ 1s^{a+3}|\log s|^{2}\,\mathrm{d}s<\infty,
\]
we get
\[
\left(\frac{C_4}{-\log\varepsilon}\right)
\int_{\varepsilon^{1/4}}^{1}
s^{a+3}|\log s|^2\,\mathrm{d}s
\to 0,
\]
as \(\varepsilon\to 0^+\). These convergences, together with the estimate \(\|\eta_{\varepsilon}\|_{L^\infty(\mathbb{R}^4)}=O\left({1}/{\sqrt{-\log\varepsilon}}\right)\) allow us to conclude that
\begin{equation}\label{convergencia-norma-zero}
    \int_{\mathbb{R}^{4}}|x|^au_{\varepsilon}^{2}\dx \to 0,
\end{equation}
as \(\varepsilon\to 0^+\). Now, noting that
\[
-\frac{|x|^{2}}{\sqrt{-8\pi^{2}\varepsilon\log\varepsilon}}
+\frac{1}{\sqrt{-8\pi^{2}\log\varepsilon}} \geq 0, \quad \mbox{if} \  \ |x|<\varepsilon^{1/4},
\]
defining \(
\widetilde{u}_{\varepsilon}
:=
{u_{\varepsilon}}/{\|u_{\varepsilon}\|_a}\), we obtain
\[
\widetilde{u}_{\varepsilon}(x)
\ge
\frac{\sqrt{\dfrac{-\log\varepsilon}{32\pi^2}}}
{\|u_{\varepsilon}\|_a},  \quad \mbox{if} \  \ |x|<\varepsilon^{1/4}.
\]
Thus, we can estimate
\[
\begin{aligned}
    \int_{\mathbb{R}^{4}} |x|^{b}\Phi_{\alpha}(\widetilde{u}_{\varepsilon})\dx
&\geq
\int_{B_{\varepsilon^{1/4}}}
|x|^{b}\Phi_{\alpha}(\widetilde{u}_{\varepsilon})\dx
\\
&=\int_{B_{\varepsilon^{1/4}}}
|x|^{b}e^{\alpha\widetilde{u}_{\varepsilon}^{2}}\dx
-
\sum_{j=0}^{j_{0}-1}
\frac{\alpha^{j}}{j!}
\int_{B_{\varepsilon^{1/4}}}
|x|^{b}
|\widetilde{u}_{\varepsilon}|^{2j}\dx.
\end{aligned}
\]
By \eqref{Qbola}, it follows that
\[
\int_{B_{\varepsilon^{1/4}}}
|x|^{b}
|\widetilde{u}_{\varepsilon}|^{2j}\dx \leq \int_{B_1}
|x|^{b}
|\widetilde{u}_{\varepsilon}|^{2j}\dx \leq C_6\|\widetilde{u}_{\varepsilon}\|_a^{2j}=C_6,
\]
for each fixed \(1\leq j\leq j_{0}-1\). If $j=0$, then we use the fact that $|x|^{b}\in L^{1}(B_{1})$. As a consequence,
\[
\begin{aligned}
    \int_{\mathbb{R}^4}|x|^{b}\Phi_\alpha(\widetilde{u}_{\varepsilon})\dx + C_7 &\geq \int_{B_{\varepsilon^{1/4}}}|x|^be^{{-\alpha\log\varepsilon}/
{32\pi^2\|u_{\varepsilon}\|^2}}\dx
\\
&=\frac{\omega_3}{b+4}e^{{-\alpha\log\varepsilon}/
{32\pi^2\|u_{\varepsilon}\|^2}}\varepsilon^{(b+4)/4}
\\
&=\frac{\omega_3}{b+4}\varepsilon^{(b+4)/4-\alpha/32\pi^2\|u_\varepsilon\|_a^2}.
\end{aligned}
\]
Since $\alpha>32\pi^2(1+b/4)$, it follows from \eqref{laplaciano-L2-zero} and \eqref{convergencia-norma-zero} that \((b+4)/4-\alpha/32\pi^2\|u_\varepsilon\|_a^2<0\), for \(\varepsilon>0\) sufficiently small. Hence
\[
\lim_{\varepsilon\to0^{+}}\int_{\mathbb{R}^4}|x|^{b}\Phi_\alpha(\widetilde{u}_{\varepsilon})\dx = \infty,
\]
whenever $\alpha>32\pi^2(1+b/4)$, completing the proof.
\end{proof}

Next, we prove the version of the weighted Adams-type inequality for the critical case. 

\begin{proof}[\textbf{Proof of Proposition \ref{caso-critico}}]
In fact, using Young’s inequality, we can estimate
    \[
    u^2 = v^2+ 2vu(R) + u^2(R) \leq v^2 + (vu(R))^{2} + 1 + u^2(R) = \widetilde{v}^2+1+u^2(R),
    \]
    where \(v\) is defined in \eqref{definicao-v} and
    \[
    \widetilde{v}:=\left(1+u^2(R)\right)^{1/2}v \in H^1_0(B_R) \cap H^2(B_R).
    \]
As a consequence,
\begin{equation}\label{est-exp-u-v}
\int_{B_R}|x|^{b}\Phi_\alpha(u)\dx \leq \int_{B_R}|x|^{b} e^{\alpha u^2}\dx \leq e^{\alpha(1+u^2(R))}\int_{B_R}|x|^{b} e^{\alpha \widetilde{v}^2}\dx.    
\end{equation}
Now, consider $\|u\|_a \leq 1$. Thus,
\begin{equation}\label{norma-laplaciano-v}
    \begin{aligned}
        \int_{B_R}|\Delta \widetilde{v}|^2\dx &= \left(1+u^2(R)\right)\int_{B_R}|\Delta {v}|^2\dx \\
        &= \left(1+u^2(R)\right)\int_{B_R}|\Delta {u}|^2\dx 
        \\
        &\leq \left(1+u^2(R)\right) \left(1-\|u\|_{L^2(\mathbb{R}^4,|x|^a)}^2 \right) 
        \\
        &\leq 1 - \|u\|_{L^2(\mathbb{R}^4,|x|^a)}^2 +u^2(R).
    \end{aligned}
\end{equation}
   Since \(u \in \mathcal{D}\), we have that 
   \[
\|u\|_{L^2(\mathbb{R}^4,|x|^a)}^2 \geq \int_{B_R}|x|^au^2\dx = 2\pi ^2\int_0^R s^au^2(s)s^3 \, \mathrm{d}s \geq \frac{2\pi^2}{(a+4)}u^2(R)R^{a+4},
\]
that is,
\[
u^2(R) \leq \frac{(a+4)}{2\pi^2 R^{a+4}}\|u\|_{L^2(\mathbb{R}^4,|x|^a)}^2.
\]
Choosing \(R>0\) large such that \((a+4)/2\pi^2 R^{a+4} \leq 1\), by \eqref{norma-laplaciano-v} we can conclude that
\[
\int_{B_R}|\Delta \widetilde{v}|^2\dx \leq 1.
\]
Defining \(\omega(y):=\widetilde{v}(Ry) \in H^1_0(B_1) \cap H^2(B_1)\), we get \(\|\Delta \omega\|_{L^2(B_1)} = \|\Delta \widetilde{v}\|_{L^2(B_R)} \leq 1\). On the other hand, from \eqref{est-exp-u-v}, we infer
\[
\int_{B_R}|x|^{b}\Phi_\alpha(u)\dx \leq C_1 \int_{B_1}|x|^{b}e^{\alpha \omega^2}\dx.
\]
Applying \cite[Theorem 1.2]{Lam-Lu} in the case \(-4<b \leq0\), and \cite[Theorem 1]{Calanchi-Tarsi} in the case \(b>0\), it follows that
\[
\sup_{\left\lbrace u \in \mathcal{D} \ : \ \|u\|_a \leq 1\right\rbrace }\int_{B_R}|x|^{b}\Phi_\alpha(u)\dx<\infty,
\]
for \(\alpha = 32\pi^2(1+b/4)\). The estimate over \(\mathbb{R}^4 \setminus B_R\) follows exactly as in the proof of Theorem \ref{TM}. This completes the proof.
\end{proof} 

\section{The problem \eqref{P}: existence, regularity, and nonexistence}\label{s4}

The proof of Theorem \ref{T1} is divided into three subsections. We first address the existence of solutions by establishing an appropriate version of the of the Principle of Symmetric Criticality and applying the Mountain Pass Theorem.

\subsection{Existence of weak solutions for Problem \eqref{P}}

The energy functional associated with \eqref{P} is given by
$$
I_\lambda(u)=\frac{1}{2}\|u\|_a^2-\lambda\int_{\mathbb{R}^4}|x|^{b}F(u)\,\mathrm{d}x, \quad  u \in E^a_{\mathrm{rad}}.
$$
To see that this functional is well-defined, note that, from \eqref{critical} and \ref{f1}, we can conclude that for each $\varepsilon > 0$, $\alpha > 32\pi^2(1+b/4)$, and $q \geq 1$, there exists \(C_f>0\) such that
\begin{equation}\label{des-F}
|f(s)|\le \varepsilon |s|^{\gamma_{a,b}-1}+C_f|s|^{q-1}\Phi_{\alpha}(s)
\ \ \text{and} \ \
|F(s)|\le \varepsilon |s|^{\gamma_{a,b}}+C_f|s|^{q}\Phi_{\alpha}(s),
\quad s\in\mathbb{R}.
\end{equation}
On the other hand, it follows from \cite[Lemma 2.1]{Yang} that
\begin{equation}\label{elementar}
\left[\Phi_{\alpha}(s) \right]^r\leq\Phi_{r\alpha}(s), \quad \mbox{for all} \ r\geq1.
\end{equation}
Combining this estimate with \eqref{des-F} (with $q \geq \gamma_{a,b}$), H\"{o}lder's inequality, and Theorems \ref{Teoimersao} and \ref{TM}, we obtain
\begin{equation}\label{est-xb-Fu}
\begin{aligned}
    \int_{\mathbb{R}^4}|x|^{b}F(u)\,\mathrm{d}x&\leq\varepsilon\|u\|_{L^{\gamma_{a,b}}(\mathbb{R}^4,|x|^{b})}^{\gamma_{a,b}}+C_f\|u\|_{L^{r_1q}(\mathbb{R}^4,|x|^{b})}^{q}\left(\int_{\mathbb{R}^4}|x|^{b}\Phi_{r_2\alpha}(u)\,\mathrm{d}x\right)^{1/r_2}
    \\
    &\leq \varepsilon C_1\|u\|_a^{\gamma_{a,b}} + C_1\|u\|_a^q\left(\int_{\mathbb{R}^4}|x|^{b}\Phi_{r_2\alpha}(u)\,\mathrm{d}x\right)^{1/r_2} <\infty,
\end{aligned}
\end{equation}
with $1/r_1 + 1/r_2 = 1$ and \(u \in E^a_{\mathrm{rad}}\). Therefore, \(I_\lambda\) is well defined. Finally, by standard arguments $I_\lambda\in C^1(E^a_{\mathrm{rad}},\mathbb{R})$ with 
	$$
    I_\lambda'(u)v=\int_{\mathbb{R}^4} \left(\Delta u \Delta v + |x|^{a}uv\right)\,\mathrm{d}x-\lambda\int_{\mathbb{R}^4}  |x|^{b}f(u)v\,\mathrm{d}x, \quad  u,v\in E^a_{\mathrm{rad}}.
	$$
It is important to note that, if \(u \in E^a_{\mathrm{rad}}\) is a critical point of \(I_\lambda\), then \(I'_\lambda(u)v=0\), for any \(v \in E^a_{\mathrm{rad}}\), but not necessarily this is valid in \(E^a \setminus E^a_{\mathrm{rad}}\), and so we cannot guarantee that \(u\) is a weak solution of \eqref{P}. Therefore, we establish below a suitable version of the Principle of Symmetric Criticality (see \cite{Palais,Willem}).

\begin{lemma}\label{criticalidade} 
	Suppose that $a,b>-4$, \eqref{critical}, and \ref{f1} hold. If $u\in E^a_{\mathrm{rad}}$ is a critical point of $I_\lambda$, then $u$ is a weak solution of \eqref{P}.
\end{lemma}

\begin{proof} Consider the linear functional $T_u:E^a\to \mathbb{R}$ given by
$$
	T_u(w):=\int_{\mathbb{R}^4} \left(\Delta u \Delta w + |x|^{a}uw\right)\,\mathrm{d}x-\lambda \int_{\mathbb{R}^4} |x|^{b}f(u)w\,\mathrm{d}x, \quad w\in E^a.
	$$
Assume for the moment that \(T_u\) is continuous, that is, there exists \(C_0>0\) such that
\begin{equation}\label{funcional-continuo}
    |T_u(w)| \leq C_0\|w\|_a, \quad w \in E^a.
\end{equation}
Thus, in light of \cite[Lemme 14.5]{Willem2}, uniform convexity implies that there exists a unique \(\widetilde{u} \in E^a\), with \(\|\widetilde{u}\|_a=1\), such that 
\begin{equation}\label{unicidade-norma}
    T_u(\widetilde{u}) = \max_{\|w\|_a=1} |T_u(w)|.
\end{equation}
Since \(
gv(x):=v(g^{-1}x)
\), for any \(g \in \mathcal{O}(4)\), observing that
\[
\begin{aligned}
    T_u(g\widetilde{u}) &= \int_{\mathbb{R}^4} \left(\Delta (gu) \Delta (g\widetilde{u}) + |g^{-1}x|^{a}(gu)(g\widetilde{u})\right)\,\mathrm{d}x-\lambda \int_{\mathbb{R}^4} |g^{-1}x|^{b}f(gu)(g\widetilde{u})\,\mathrm{d}x,
\end{aligned}
\]
we can apply the Change of Variable Theorem together with \(g^{-1}\mathbb{R}^4 = \mathbb{R}^4\) to get \(T_u(g\widetilde{u}) = T_u(\widetilde{u})\). Similarly, \(\|g\widetilde{u}\|_a=\|\widetilde{u}\|_a=1\). By uniquenes in \eqref{unicidade-norma}, we can conclude that \(g\widetilde{u}=\widetilde{u}\), and so \(\widetilde{u} \in E^a_{\mathrm{rad}}\). Therefore,
\[
\max_{\|w\|_a=1} |T_u(w)| = T_u(\widetilde{u})=0,
\]
because \(u\) is a critical point of \(I_\lambda\) restricted to \(E^a_{\mathrm{rad}}\). Thus, if \(w\in E^a\setminus \{0\}\), then
\[
T_u(w) =   \|w\|_aT_u\left(\frac{w}{\|w\|_a}\right) =0,
\]
  and therefore \(u\) is a weak solution of \eqref{P}. Consequently, it remains to prove \eqref{funcional-continuo}. To this end, by H\"{o}lder's inequality, we have
  \begin{equation}\label{est-holder-laplaciano}
      \left| \int_{\mathbb{R}^4} \Delta u \Delta w \,\mathrm{d}x\right| \leq \left(\int_{\mathbb{R}^4} |\Delta u|^2\,\mathrm{d}x\right)^{1/2} \left(\int_{\mathbb{R}^4} |\Delta w|^2\,\mathrm{d}x\right)^{1/2}\leq C_1\|w\|_a,
  \end{equation}
    and 
    \begin{equation}\label{est-holder-xa}
        \left| \int_{\mathbb{R}^4} |x|^{a} u w \,\mathrm{d}x\right| \leq \left(\int_{\mathbb{R}^4} |x|^{a} u^2 \,\mathrm{d}x\right)^{1/2} \left(\int_{\mathbb{R}^4} |x|^{a} w^2 \,\mathrm{d}x\right)^{1/2}\leq C_2\|w\|_a.
    \end{equation}
	Now, we estimate the term \( \displaystyle \left|\int_{\mathbb{R}^4} |x|^{b}f(u)w\,\mathrm{d}x\right|\). From \eqref{des-F}, there holds
    \begin{equation}\label{estimativa-termo-f}
        \left|\int_{\mathbb{R}^4}|x|^{b}f(u)w\,\mathrm{d}x\right| \leq\varepsilon\int_{\mathbb{R}^4}|x|^{b}|u|^{\gamma_{a,b}-1}|w|\,\mathrm{d}x+C_f\int_{\mathbb{R}^4}|x|^{b}|u|^{q-1}\Phi_{\alpha}(u)|w|\,\mathrm{d}x.
    \end{equation}
For \(R>0\), we can write
    \begin{equation}\label{q1}
		\int_{\mathbb{R}^4}|x|^{b}|u|^{\gamma_{a,b}-1}|w|\,\mathrm{d}x=\int_{B_R}|x|^{b}|u|^{\gamma_{a,b}-1}|w|\,\mathrm{d}x+\int_{\mathbb{R}^4 \setminus B_R}|x|^{b}|u|^{\gamma_{a,b}-1}|w|\,\mathrm{d}x.
	\end{equation}
Choosing \(r_1>1\) such that $|x|^{r_1b} \in L^1(B_R)$, we can apply H\"{o}lder's inequality and Lemma~\ref{lema-imersao-local} to get
\begin{equation}\label{integral-bola-xb}
    \begin{aligned}
\int_{B_R}|x|^{b}|u|^{\gamma_{a,b}-1}|w|\,\mathrm{d}x &\leq \left(\int_{B_R}|x|^{r_1b}\,\mathrm{d}x\right)^{1/r_1}\|u\|_{L^{r_2(\gamma_{a,b}-1)}(B_R)}^{\gamma_{a,b}-1}\|w\|_{L^{r_3}(B_R)} \\
&\leq C_4\|u\|_a^{\gamma_{a,b}-1}\|w\|_a,   
\end{aligned}
\end{equation}
with $1/r_1+1/r_2+1/r_3=1$. To estimate the integral over \(\mathbb{R}^4 \setminus B_R\), we can use Lemma \ref{lema-radial} to obtain
\[
|u(|x|)|^{\gamma_{a,b}-2} \leq C_5|x|^{-(\gamma_{a,b}-2)(a+4)/4}\|u\|_a^{\gamma_{a,b}-2}, \quad |x|\geq R.
\]
As a consequence,
\[
\int_{\mathbb{R}^4 \setminus B_R}|x|^{b}|u|^{\gamma_{a,b}-1}|w|\,\mathrm{d}x \leq C_6\|u\|_a^{\gamma_{a,b}-2}\int_{\mathbb{R}^4 \setminus B_R}|x|^\xi|x|^a|u||w|\dx,
\]
with \(\xi:=b-a-(\gamma_{a,b}-2)(a+4)/4\). By definition of \(\gamma_{a,b}\) (see \eqref{gamma}), it holds that \(\xi\leq0\). This, combined with H\"{o}lder's inequality, implies that
\[
\int_{\mathbb{R}^4 \setminus B_R}|x|^{b}|u|^{\gamma_{a,b}-1}|w|\,\mathrm{d}x \leq C_6\|u\|_a^{\gamma_{a,b}-2}R^\xi\|u\|_a\|w\|_a = C_7\|w\|_a,
\]
for some \(C_7=C_7(u)>0\). From the last estimate, together with \eqref{q1} and \eqref{integral-bola-xb}, we have
\[
\int_{\mathbb{R}^4}|x|^{b}|u|^{\gamma_{a,b}-1}|w|\,\mathrm{d}x \leq C_8\|w\|_a.
\]
Applying this inequality in \eqref{estimativa-termo-f}, it holds that
\begin{equation}\label{f(u)w}
\left|\int_{\mathbb{R}^4}|x|^{b}f(u)w\,\mathrm{d}x\right| \leq \varepsilon C_8\|w\|_a+C_f\int_{\mathbb{R}^4}|x|^{b}|u|^{q-1}\Phi_{\alpha}(u)|w|\,\mathrm{d}x.
\end{equation}
For the local part, using H\"{o}lder's inequality (with \(1/t_1+1/t_2=1\)), \eqref{elementar}, and Theorem \ref{TM}, one infers
\[
\begin{aligned}
    \int_{B_R}|x|^{b}|u|^{q-1}\Phi_{\alpha}(u)|w|\,\mathrm{d}x &\leq \left(\int_{B_R}|x|^b\Phi_{t_1\alpha}(u)\dx\right)^{1/t_1} \left(\int_{B_R}|x|^b(|u|^{q-1}|w|)^{t_2}\dx\right)^{1/t_2}
    \\
    &=C_9\left(\int_{B_R}|x|^b(|u|^{q-1}|w|)^{t_2}\dx\right)^{1/t_2}.
\end{aligned}
\]
Again, H\"{o}lder's inequality, with \(1/s_1+1/s_2+1/s_3=1\) and \(s_1>1\) satisfying \(s_1b>-4\), yields
\[
\int_{B_R}|x|^b(|u|^{q-1}|w|)^{t_2}\dx \leq \left(\int_{B_R}|x|^{s_1b}\dx\right)^{1/s_1}\|u\|_{L^{s_2t_2(q-1)}(B_R)}^{q-1}\|w\|_{L^{s_3t_2}(B_R)}^{t_2} \leq C_{10}\|w\|_a^{t_2},
\]
where we used Lemma \ref{lema-imersao-local}. Combining the last two inequalities, we have that
\begin{equation}\label{est-bola-exp-xb}
    \int_{B_R}|x|^{b}|u|^{q-1}\Phi_{\alpha}(u)|w|\,\mathrm{d}x \leq C_{11}\|w\|_a.
\end{equation}
For the exterior part, by H\"{o}lder's inequality, \eqref{elementar}, and Theorem~\ref{TM}, we obtain
\begin{equation}\label{est-fora-da-bola-exp}
    \int_{\mathbb{R}^4 \setminus B_R}|x|^{b}|u|^{q-1}\Phi_{\alpha}(u)|w|\,\mathrm{d}x \leq C_{12}\left(\int_{\mathbb{R}^4 \setminus B_R}|x|^b|u|^{2(q-1)}|w|^2\,\mathrm{d}x\right)^{1/2},
\end{equation}
where
\( \displaystyle
C_{12}:=\left(\int_{\mathbb{R}^4 \setminus B_R}|x|^{b}\Phi_{2\alpha}(u)\,\mathrm{d}x\right)^{1/2}
\). In view of Lemma \ref{lema-radial}, there holds
	$$
	\begin{aligned}
		\int_{\mathbb{R}^4 \setminus B_R}|x|^b|u|^{2(q-1)}|w|^2\,\mathrm{d}x &=\int_{\mathbb{R}^4 \setminus B_R}|x|^{b-a}|u|^{2(q-1)}|x|^a|w|^2\,\mathrm{d}x
		\\
		&\leq C_{13}\int_{\mathbb{R}^4 \setminus B_R}|x|^{b-a-(q-1)(a+4)/2}|x|^a|w|^2\,\mathrm{d}x,
	\end{aligned}
	$$
    for some \(C_{13}=C_{13}(u)>0\). Choosing \(q\geq1\) such that $b-a-(q-1)(a+4)/2 \leq 0 $, we get
\[
\int_{\mathbb{R}^4 \setminus B_R}|x|^b|u|^{2(q-1)}|w|^2\,\mathrm{d}x \leq C_{14}\|w\|_a^2.
\]
Using this estimate together with \eqref{est-fora-da-bola-exp} and the above inequality, we have
\[
 \int_{\mathbb{R}^4 \setminus B_R}|x|^{b}|u|^{q-1}\Phi_{\alpha}(u)|w|\,\mathrm{d}x \leq C_{15}\|w\|_a.
\]
Combining this with \eqref{est-holder-laplaciano}, \eqref{est-holder-xa}, \eqref{f(u)w}, and \eqref{est-bola-exp-xb}, we deduce
$$
    |T_u(w)|= \left| \int_{\mathbb{R}^4} (\Delta u \Delta w + |x|^{a}uw)\,\mathrm{d}x-\lambda \int_{\mathbb{R}^4}|x|^{b}f(u)w\,\mathrm{d}x \right| \leq C_{16} \| w\|_a, 
    $$
   which shows that $T_u$ is continuous, completing the proof.
   \end{proof}

    The mountain pass geometry for the functional \(I_\lambda\) is provided in the next lemma.

    \begin{lemma}\label{passo-da-montanha} Suppose that $a,b>-4$, $\lambda>0$, and that \eqref{critical}, \ref{f1}, and \ref{f2} hold. Then, there exist $\tau,\rho>0$ such that $I_\lambda(u)\geq\tau$, whenever $\|u\|_a=\rho$. Furthermore, there exists $e\in E^a_{\mathrm{rad}}$ satisfying $\|e\|_a>\rho$ and $I_\lambda(e)<0$.
\end{lemma}

\begin{proof}
In view of \eqref{est-xb-Fu} and the property 
\begin{equation}\label{propriedade-phi}
    \Phi_\alpha(ts) = \Phi_{\alpha t^2}(s),
\end{equation}
we obtain
\[
\int_{\mathbb{R}^4}|x|^{b}F(u)\,\mathrm{d}x \leq \varepsilon C_1\|u\|_a^{\gamma_{a,b}} + C_1\|u\|_a^q\left(\int_{\mathbb{R}^4}|x|^{b}\Phi_{r_2\alpha\|u\|_a^2}\left(\frac{u}{\|u\|_a}\right)\,\mathrm{d}x\right)^{1/r_2}.
\]
If $\|u\|_a \leq\rho$, with \(r_2\alpha\rho^2<32\pi^2(1+b/4)\), by Theorem \ref{TM},
$$
	\int_{\mathbb{R}^4}|x|^{b}F(u)\,\mathrm{d}x\leq\varepsilon C_2\|u\|_a^{\gamma_{a,b}}+C_2\|u\|_a^q.
	$$
So, if $\|u\|_a\leq\rho$, then
$$
	I_\lambda(u)\geq\frac{1}{2}\|u\|_a^2-\lambda(\varepsilon C_2\|u\|_a^{\gamma_{a,b}}+C_2\|u\|_a^q)=\|u\|_a^2\left(\frac{1}{2}-\varepsilon \lambda C_2\|u\|_a^{\gamma_{a,b}-2}-\lambda C_2\|u\|_a^{q-2}\right).
	$$
    Choosing $0<\varepsilon<1/(4\lambda C_2)$  and $q>\gamma_{a,b}\geq 2$, we may choose \(\rho>0\) sufficiently small such that
$
\lambda C_{2}\rho^{q-2}\leq 1/8
$
and, whenever \(\gamma_{a,b}>2\),
$
\varepsilon\lambda C_{2}\rho^{\gamma_{a,b}-2}\leq 1/8.
$
If \(\gamma_{a,b}=2\), the corresponding term is already controlled by the choice of
\(\varepsilon\). Therefore, for every \(u\in E_{\mathrm{rad}}^{a}\) satisfying
\(\|u\|_{a}=\rho\), we obtain
\[
I_{\lambda}(u)
\geq
\|u\|_{a}^{2}
\left(
\frac12
-\varepsilon\lambda C_{2}\|u\|_{a}^{\gamma_{a,b}-2}
-\lambda C_{2}\|u\|_{a}^{q-2}
\right)
\geq
\frac18\|u\|_{a}^{2}
=
\frac{\rho^{2}}{8}.
\]

Finally, let $\varphi\in C_{0,\mathrm{rad}}^\infty(\mathbb{R}^4)\backslash\{0\}$. By \ref{f2}, there exist constants $C_3,C_4>0$ such that $F(s)\geq C_3|s|^\theta-C_4$, for all $s\in\mathbb{R}$. Consequently, for any $t>0$,
$$
	I_\lambda(t\varphi)\leq\frac{t^2}{2}\|\varphi\|_a^2-C_3 \lambda t^\theta\int_{\mathrm{supp \, (\varphi)}}|x|^{b}|\varphi|^\theta\,\mathrm{d}x+C_4 \lambda \int_{\mathrm{supp \, (\varphi)}}|x|^{b}\,\mathrm{d}x.
	$$
    Since $\theta>\gamma_{a,b}\geq 2$, we can choose $e:=t\varphi$, for $t>0$ sufficiently large satisfying $\|e\|_a>\rho$ and $I_\lambda(e)<0$. This completes the proof.
    \end{proof}

Now, recall that a sequence $(u_{n})_{n}\subset E^a_{\mathrm{rad}}$ is called a $(PS)_{c}$ sequence if
\[
I_\lambda(u_{n}) \to c \quad \mbox{and} \quad \Vert I_\lambda^{\prime}(u_{n})\Vert\rightarrow0, \quad \mbox{as} \ \  n\to \infty.
\]
In this case, the functional $I_\lambda$ is said to satisfy the $(PS)_{c}$ condition if every $(PS)_{c}$ sequence admits a convergent subsequence. In the next result, we show that there exists a constant $c_{0}>0$ such that if $c\in (0,c_{0})$, then $I_\lambda$ satisfies  the $(PS)_{c}$ condition.

\begin{lemma}\label{condicao-PS}
	Suppose that $a,b>-4$, \eqref{critical}, $\lambda>0$, and that \ref{f1}, and \ref{f2} hold. If 
    \[
    0<c<\frac{\theta-2}{2\theta},
    \]
    then the functional \(I_\lambda\) satisfies  the $(PS)_{c}$ condition.
\end{lemma}

\begin{proof}
	Let $(u_n)_{n}\subset E^a_{\mathrm{rad}}$ be a $(PS)_c$ sequence. Hence, 
\begin{equation}\label{ps}
	I_\lambda(u_n)=c+o_n(1) \quad \mbox{and} \quad I_\lambda'(u_n)u_n=o_n(\|u_n\|_a).
	\end{equation} 
Together with \ref{f2}, these convergences imply that there exist $C_1,C_2>0$ such that
    \begin{equation}\label{limitada}
	\begin{aligned}
		C_1+C_2\|u_n\|_a &\geq I_\lambda(u_n)-\frac{1}{\theta}I_\lambda'(u_n)u_n\\
		&=\left(\frac{1}{2}-\frac{1}{\theta}\right)\|u_n\|_a ^2-\lambda\int_{\mathbb{R}^4}  |x|^{b}\left(F(u_n)-\frac{1}{\theta}f(u_n)u_n\right)\,\mathrm{d}x\\
		&\geq\left(\frac{1}{2}-\frac{1}{\theta}\right)\|u_n\|_a^2.
	\end{aligned}
	\end{equation}
Thus, $(u_n)_{n}$ is bounded in $E^a_{\mathrm{rad}}$. Therefore, up to a subsequence, $u_n \rightharpoonup u$ weakly in $E^a_{\mathrm{rad}}$.  We claim that
	\begin{equation}\label{convergencia-fun}
	\int_{\mathbb{R}^4}|x|^{b}f(u_n)(u_n-u)\,\mathrm{d}x=o_n(1).
	\end{equation}
Once \eqref{convergencia-fun} is proved, since $I_\lambda'(u_n)(u_n-u)=o_n(1)$, one has
$$
	\langle u_n, u_n-u \rangle_a =\int_{\mathbb{R}^4} \left(\Delta u_n \Delta (u_n-u) + |x|^{a}u_n(u_n-u)\right)\,\mathrm{d}x =o_n(1).
	$$
On the other hand, the weak convergence implies that
    $$
    \langle u, u_n-u \rangle_a = \int_{\mathbb{R}^4}\left(\Delta u\Delta (u_n-u)+ |x|^au(u_n-u)\right)\,\mathrm{d}x =o_n(1), 
    $$
and so \(\| u_n-u\|_a ^2 = o_n(1)\). Thus, $u_n\to u$ strongly in $E^a_{\mathrm{rad}}$. Therefore, it is sufficient to prove \eqref{convergencia-fun}. In fact, by \eqref{des-F} (with \(q=1\)),
	$$
	\left| \int_{\mathbb{R}^4}|x|^{b}f(u_n)(u_n-u)\,\mathrm{d}x\right|\leq\varepsilon A_n+C_3B_n,
	$$
	where 
	$$
	A_n:=\int_{\mathbb{R}^4}|x|^{b}|u_n|^{\gamma_{a,b}-1}|u_n-u|\,\mathrm{d}x \quad \mbox{and} \quad B_n:=\int_{\mathbb{R}^4}|x|^{b}\Phi_{\alpha}(u_n)|u_n-u|\,\mathrm{d}x.
	$$
H\"{o}lder's inequality yields
	$$
	A_n\leq\|u_n\|_{L^{\gamma_{a,b}}(\mathbb{R}^4,|x|^{b})}^{\gamma_{a,b}-1}\|u_n-u\|_{L^{\gamma_{a,b}}(\mathbb{R}^4,|x|^{b})}.
	$$
    Thus, \((A_n)_n\) is bounded in \(\mathbb R\). Since \(\varepsilon>0\) is arbitrary, it remains to prove that \(B_n=o_n(1)\). By \eqref{propriedade-phi}, we can write
    \begin{equation}\label{Bn}
        B_n = \int_{\mathbb{R}^4}|x|^{b}\Phi_{\alpha\|u_n\|_a^2}\left(\frac{u_n}{\|u_n\|_a}\right)|u_n-u|\,\mathrm{d}x.
    \end{equation}
    To estimate \(B_n\), arguing as in \eqref{limitada}, we obtain
    $$
	c=\lim_{n\to\infty}\left(I_\lambda(u_n)-\frac{1}{\theta}I_\lambda'(u_n)u_n\right)\geq\limsup_{n\to\infty}\left(\frac{1}{2}-\frac{1}{\theta}\right)\|u_n\|_a^2,
	$$
that is,
\[
\limsup_{n\to\infty}\|u_n\|_a^2 \leq\frac{2\theta}{\theta-2}c.
\]
Hence, assuming \(c<(\theta-2)/2\theta\), there holds
    $$
	\limsup_{n\to\infty}\|u_n\|_a^{2}<1.
	$$
Thus, we can choose 
\[
32\pi^2\left(1+\frac{b}{4}\right)<\alpha<32\pi^2\left(1+\frac{b}{4}\right)\left[\limsup_{n\to\infty}\|u_n\|_a^{2}\right]^{-1},
\]
if $\limsup_{n\to\infty}\|u_n\|_a^{2}>0$. Taking \(r_1\in \mathbb{R}\) such that
\[
1<r_1<32\pi^2\left(1+\frac{b}{4}\right)\left[\alpha\limsup_{n\to\infty}\|u_n\|_a^{2}\right]^{-1},
\]
one concludes that $r_1\alpha\|u_n\|_a^{2}<32\pi^2(1+b/4)$, for $n\in\mathbb{N}$ large enough. The case $\lim_{n\to\infty}\|u_n\|_a^{2}=0$ is immediate. Next, applying H\"older's inequality in \eqref{Bn} together with
\eqref{elementar} and Theorems \ref{Teoimersao} and \ref{TM}, it follows that
\[
B_n \leq \left(\int_{\mathbb{R}^4}|x|^{b}\Phi_{r_1\alpha\|u_n\|_a^{2}}\left(\frac{u_n}{\|u_n\|_a}\right)\,\mathrm{d}x\right)^{1/r_1}\|u_n-u\|_{L^{r_2}(\mathbb{R}^4,|x|^{b})} \leq C_4o_n(1),
\]
with \(r_2>\gamma_{a,b}\). Therefore, \eqref{convergencia-fun} is true and this completes the proof.
\end{proof}

In the next Proposition, we consider $\lambda>0$ sufficiently large to control the minimax energy level of the Mountain Pass Theorem in the range of Lemma \ref{condicao-PS}, where the functional \(I_\lambda\) satisfies  the $(PS)_{c}$ condition.

\begin{proposition}\label{solucao}
Suppose that $a,b>-4$, \eqref{critical}, and \ref{f1}-\ref{f3} hold. Then, for each 
\[
\lambda >\lambda_0:= \max \left\lbrace \frac{S_\nu}{2\mu} , \frac{S_\nu^{\nu/2}}{\mu\nu}\left(\frac{\theta(\nu-2)}{\nu(\theta-2)}\right)^{(\nu-2)/2}\right\rbrace
\]
the problem \eqref{P} admits at least one weak solution \(u_\lambda \in E^a_{\mathrm{rad}} \setminus \{0\}\).
\end{proposition}

\begin{proof}
By the Mountain Pass Theorem, there exists a sequence \((u_n)_n\) in \(E^a_{\mathrm{rad}}\) such that 
\[
I_\lambda(u_{n}) \to c_\lambda \quad \mbox{and} \quad \Vert I_\lambda^{\prime}(u_{n})\Vert\rightarrow0, \quad \mbox{as} \ \  n\to \infty,
\]
where
 \[
  c_\lambda:=\inf_{g\in \Gamma_\lambda}\max_{t\in[0,1]}I_\lambda(g(t))
 \]
and
 \[
  \Gamma_\lambda:=\left\lbrace g \in C([0,1],E^a_{\mathrm{rad}}) : g(0)=0 \mbox{ and } I_\lambda(g(1))<0\rbrace\right..
 \]
In this case, Lemma \ref{passo-da-montanha} ensures that \(0<\tau\leq c_\lambda<\infty\). Combining Lemmas \ref{criticalidade} and \ref{condicao-PS}, it remains to show that
    \[
    c_\lambda<\frac{\theta-2}{2\theta}.
    \]
First, we can apply Theorem ~\ref{Teoimersao} to get \(v\in E^a_{\mathrm{rad}}\) satisfying 
    $$
	\|v\|_a^2=S_\nu:=\inf\left\lbrace \|u\|_a^2:\int_{\mathbb{R}^4}|x|^{b}|u|^\nu\,\mathrm{d}x=1 \right\rbrace.
	$$
In view of \ref{f3}, we have
    $$
	I_\lambda(v)\leq\frac{1}{2} S_\nu-\lambda\mu\int_{\mathbb{R}^4}|x|^{b}|v|^\nu \,\mathrm{d}x = \frac{1}{2}S_\nu-\lambda\mu<0,
	$$
for any $\lambda>S_\nu/2\mu$. Hence, the path \(g(t):=tv\) belongs to \(\Gamma_\lambda\). Thus,
$$
	c_\lambda\leq\max_{0\leq t\leq1}I_\lambda(tv)\leq\max_{t\geq0}\left(\frac{1}{2}S_\nu t^2-\lambda\mu t^\nu\right)=\frac{S_\nu^{\nu/(\nu-2)}}{(\lambda\mu\nu)^{2/(\nu-2)}}
	\left(\frac{\nu-2}{2\nu}\right)<\frac{\theta-2}{2\theta},
	$$
whenever \(\lambda>\lambda_0\), completing the proof.
\end{proof}

\subsection{Regularity of weak solutions}

The next result establishes the regularity of the weak solutions obtained in the preceding subsection.

\begin{lemma}\label{regularidade} Let \(u\) be a weak solution of \eqref{P}. Then,
\begin{itemize}
    \item [$(a)$] there exists \(1<\widetilde{p}<\infty\) such that \(u \in W^{4,\widetilde{p}}_{\mathrm{loc}} (\mathbb{R}^4) \cap C^{0,\widetilde{\sigma}}_{\mathrm{loc}}(\mathbb{R}^4) \cap L^\infty(\mathbb{R}^4)\), for some $\widetilde{\sigma}\in(0,1)$;

    \item[$(b)$] if \(a,b>-1\), then \(u \in C^{3,\sigma}_{\mathrm{loc}}(\mathbb{R}^4)\), for some \(\sigma\in(0,1)\);

    \item[$(c)$] if \(a,b\geq0\) and \(f\) is a Lipschitz function, then \(u \in C^{4,\bar{\sigma}}_{\mathrm{loc}}(\mathbb{R}^4)\), for some \(\bar{\sigma}\in(0,1)\).
\end{itemize}
\end{lemma}

\begin{proof} 
To prove \((a)\), we claim that \(h \in L^{\widetilde{p}}(B_R)\), for some \(1<\widetilde{p}<\infty\), where
    \[
    h(x):=\lambda|x|^bf(u(x))-|x|^au(x).
    \]
 Since \(a,b>-4\), we can choose \(\widetilde{p}>1\) such that \(\widetilde{p}a,\widetilde{p}b>-4\). Hence, we can take \(t_1>1\) satisfying the integrability condition \(|x|^{t_1\widetilde{p}a} \in L^1(B_R)\). Therefore, Hölder's inequality together with Lemma \ref{lema-imersao-local} yields
\begin{equation}\label{conta-Lp-bola-a}
    \int_{B_R}\left||x|^au(x)\right|^{\widetilde{p}} \dx \leq \left(\int_{B_R}|x|^{t_1\widetilde{p}a}\dx\right)^{1/t_1}\left(\int_{B_R}|u|^{t_2\widetilde{p}}\dx\right)^{1/t_2}<\infty,
\end{equation}
where \(1/t_1+1/t_2=1\). It remains to prove that \(|x|^bf(u(x)) \in L^{\widetilde{p}}(B_R)\). From \eqref{des-F}, with \(q=1\), one has
    \[
    |x|^b|f(u)| \leq \varepsilon |x|^b|u|^{\gamma_{a,b}-1} + C_f|x|^b\Phi_\alpha(u).
    \]
Similarly to \eqref{conta-Lp-bola-a}, we get \(|x|^b|u|^{\gamma_{a,b}-1} \in L^{\widetilde{p}}(B_R)\). Finally, by Hölder's inequality, with \(s_1>1\) such that \(s_1\widetilde{p}b>-4\) and \(1/s_1+1/s_2=1\), and \eqref{elementar} there holds
 \[
\int_{B_R}\left||x|^b\Phi_\alpha(u)\right|^{\widetilde{p}} \dx \leq  \left(\int_{B_R}|x|^{s_1\widetilde{p}b}\dx\right)^{1/s_1}\left(\int_{B_R}\Phi_{s_2\widetilde{p}\alpha}(u)\dx\right)^{1/s_2}.
\]
Applying Theorem \ref{TM}, we conclude that \(\Phi_{s_2\widetilde{p}\alpha}(u) \in L^1(\mathbb{R}^4)\). This and Lemma \ref{lema-imersao-local} together with the above estimate imply that \(|x|^b\Phi_\alpha(u) \in L^{\widetilde{p}}(B_R)\), and hence \(h \in L^{\widetilde{p}}(B_R)\). Since $R>0$ is arbitrary, \(h \in L^{\widetilde{p}}_{\mathrm{loc}}(\mathbb{R}^4)\).  As
\begin{equation}\label{laplaciano-do-laplaciano}
    -\Delta(-\Delta u) = h, \quad \mbox{in}  \  \mathbb{R}^4,
\end{equation}
by regularity theory, we can conclude that \(\Delta u \in W^{2,\widetilde{p}}_{\mathrm{loc}}(\mathbb{R}^4)\). Defining \(g:=\Delta u \), we have that
\begin{equation}\label{laplaciano-g}
    -\Delta (u_{x_ix_j}) = -g_{x_ix_j}, \quad \mbox{in}  \ \mathbb{R}^4,
\end{equation}
with \(g_{x_ix_j} \in L^{\widetilde{p}}_{\mathrm{loc}}(\mathbb{R}^4)\). Applying interior elliptic regularity once more, it follows that \(u_{x_ix_j} \in W^{2,\widetilde{p}}_{\mathrm{loc}}(\mathbb{R}^4)\), and hence \(u \in W^{4,\widetilde{p}}_{\mathrm{loc}}(\mathbb{R}^4) \). By the embedding \( W^{4,\widetilde{p}}(B_R) \hookrightarrow C^{0,\widetilde{\sigma}}(\overline{B_R}) \), for some $\widetilde{\sigma}\in(0,1)$, and Lemma \ref{lema-radial}, it holds that \(u \in L ^\infty(\mathbb{R}^4)\). This proves item $(a)$.

To prove \((b)\), we claim that there exists \(p_1>4\) such that \(u \in W^{4,p_1}(B_R)\). To use the same reasoning above, we need to verify that \(p_1a,p_1b>-4\). The case \(a,b\geq0\) is immediate. If \(-1<a,b<0\), then considering
\[
4<p_1<\frac{4}{-a} \quad \mbox{and} \quad 4<p_1<\frac{4}{-b},
\]
we obtain \(p_1a,p_1b>-4\). The argument used in item \((a)\) allows us to conclude that \(u \in W^{4,p_1}(B_R)\). Since \(p_1>4\), we have \(3<4-4/p_1\), and so \(u \in W^{4,p_1}(B_R) \hookrightarrow C^{3,\sigma}(\overline{B_R})\), with \(\sigma\in(0,1)\), proving \((b)\).

Finally, we prove \((c)\). From item \((b)\) and the fact that  \(f\) is locally Lipschitz continuous, we obtain that \(h(x)=\lambda|x|^bf(u(x))-|x|^au(x)\) is locally H\"older continuous. This condition applied in \eqref{laplaciano-do-laplaciano} allows us to conclude that \(g=\Delta u \in C^{2,\sigma_1}(\overline{B_R})\), for some \(\sigma_1\in(0,1)\). Thus, we can use \eqref{laplaciano-g} to obtain \(u_{x_ix_j} \in C^{2,\sigma_2}(\overline{B_R}) \), with \(\sigma_2\in(0,1)\). Therefore, \(u \in C^{4,\sigma_3}(\overline{B_R})\), where \(\sigma_3\in (0,1)\), completing the proof.
\end{proof}

\subsection{A Pohozaev identity and nonexistence of solutions}

This subsection is devoted to the derivation of a Pohozaev identity for problem \eqref{P}. Since $u \in W^{4,{p}}_{\mathrm{loc}} (\mathbb{R}^4) $, we have
\[
\int_{\mathbb{R}^4}\Delta u \Delta \psi \dx = \int_{\mathbb{R}^4}\Delta^2 u \psi\dx, \quad \mbox{for all } \psi\in C_0^\infty(\mathbb{R}^4).
\]
Together with \eqref{weak}, this implies that
\[
\int_{\mathbb{R}^4}\left[ \Delta^2u + |x|^au - \lambda |x|^{b}f(u)\right]\psi \dx=0 , \quad \mbox{for all } \psi\in C_0^\infty(\mathbb{R}^4),
\]
and so
\[
\Delta^2u +|x|^{a}u = \lambda |x|^{b}f(u), \quad \mbox{a.e. in}  \  \mathbb{R}^4.
\]
Consider \(\varphi \in C^\infty(\mathbb{R})\) such that 
\[
\varphi \equiv 1, \ \ \mbox{in} \ (-\infty,1], \quad \varphi \equiv 0, \ \ \mbox{in} \ [2,\infty), \quad \mbox{and} \quad 0\leq \varphi\leq 1.
\]
Next, for each \(n\in\mathbb{N}\) define \(\varphi_n(x):=\varphi\left(|x|^2/n^2 \right) \in C_0^\infty(\mathbb{R}^4) \). As a consequence,
\begin{equation}\label{solucao-a.e.}
    \left(\Delta^2u +|x|^{a}u \right)\varphi_n [x\cdot \nabla u] = \lambda |x|^{b}f(u)\varphi_n [x\cdot \nabla u], \quad \mbox{a.e. in}  \  \mathbb{R}^4.
\end{equation}
Noting that
\[
\begin{aligned}
    \left(\varphi_n|x|^b F(u) x_i\right)_{x_i} &= (\varphi_n)_{x_i}|x|^b F(u) x_i + b\varphi_n|x|^{b-2}x_i^2F(u) + \varphi_n|x|^bf(u)u_{x_i}x_i
    \\
    &\quad + \varphi_n|x|^b F(u),
\end{aligned}
\]
we get
\[
\mathrm{div}\left(\varphi_n |x|^bF(u)x\right) = |x|^bF(u)[x\cdot\nabla \varphi_n] + (b+4)\varphi_n|x|^b F(u) + \varphi_n|x|^bf(u)[x\cdot\nabla u].
\]
Replacing \(F(u)\) by \(u^2/2\), we deduce
\[
\mathrm{div}\left(\varphi_n |x|^a\frac{u^2}{2}x\right) = |x|^a\frac{u^2}{2}[x\cdot\nabla \varphi_n] + (a+4)\varphi_n|x|^a \frac{u^2}{2} + \varphi_n|x|^au[x\cdot\nabla u].
\]
On the other hand, we can write
\[
\operatorname{div}\big(\varphi_n [x \cdot \nabla u]\nabla (\Delta u)\big) = [x \cdot \nabla u](\nabla \varphi_n \cdot \nabla (\Delta u)) + \varphi_n \nabla [x \cdot \nabla u] \cdot \nabla (\Delta u) + \Delta^2 u\, \varphi_n [x \cdot \nabla u].
\]
Applying the last three identities to \eqref{solucao-a.e.}, we obtain
\[
\begin{aligned}
\operatorname{div}\big(\varphi_n [x \cdot \nabla u]\nabla (\Delta u)\big) &- [x \cdot \nabla u](\nabla \varphi_n \cdot \nabla (\Delta u)) - \varphi_n \nabla [x \cdot \nabla u] \cdot \nabla (\Delta u)    
\\
&\quad + \mathrm{div}\left(\varphi_n |x|^a\frac{u^2}{2}x\right) - |x|^a\frac{u^2}{2}[x\cdot\nabla \varphi_n] - (a+4)\varphi_n|x|^a \frac{u^2}{2}
\\
&=\lambda\mathrm{div}\left(\varphi_n |x|^bF(u)x\right) - \lambda|x|^bF(u)[x\cdot\nabla \varphi_n] 
\\
&\quad-\lambda(b+4)\varphi_n|x|^b F(u).
\end{aligned}
\]
Hence,
\begin{equation}\label{conta-divergente}
    \begin{aligned}
&\mathrm{div}\left(\varphi_n\left([x \cdot \nabla u]\nabla (\Delta u) + |x|^a\frac{u^2}{2}x-\lambda|x|^bF(u)x  \right)\right) = [x \cdot \nabla u](\nabla \varphi_n \cdot \nabla (\Delta u)) \\
&\quad+ \varphi_n \nabla [x \cdot \nabla u] \cdot \nabla (\Delta u)  + |x|^a\frac{u^2}{2}[x\cdot\nabla \varphi_n] + (a+4)\varphi_n|x|^a \frac{u^2}{2}
\\
&\quad- \lambda|x|^bF(u)[x\cdot\nabla \varphi_n] -\lambda(b+4)\varphi_n|x|^b F(u).
\end{aligned}
\end{equation}
Fix \(n\in \mathbb{N}\) and consider \(R>0\) such that 
\(\mathrm{supp} \, (\varphi_n) \subset B_R\). Recall that, for \(\omega\in W^{1,p}_{loc}(\mathbb R^4)\), we have
\[
\int_{B_R}\omega \frac{\partial \varphi_n}{\partial x_i} \dx = - \int_{B_R} \frac{\partial \omega}{\partial x_i}\varphi_n \dx, 
\]
for \(i \in \{1,2,3,4\}\), we can conclude 
\begin{equation}\label{divergente-zero}
\int_{B_R}\frac{\partial}{\partial x_i}( \varphi_n  \omega)\dx =0.    
\end{equation}
Applying this equality to \eqref{conta-divergente}, we obtain
\begin{equation}\label{conta-div-final}
    \begin{aligned}
    0&=\int_{B_R}[x \cdot \nabla u](\nabla \varphi_n \cdot \nabla (\Delta u))\dx + \int_{B_R}\varphi_n \nabla [x \cdot \nabla u] \cdot \nabla (\Delta u)\dx  
    \\
    &\quad+ \frac{1}{2}\int_{B_R}|x|^au^2[x\cdot\nabla \varphi_n]\dx 
     + \frac{(a+4)}{2}\int_{B_R}\varphi_n|x|^a u^2\dx
     \\
     &\quad -\lambda\int_{B_R}|x|^bF(u)[x\cdot\nabla \varphi_n]\dx -\lambda(b+4)\int_{B_R}\varphi_n|x|^b F(u)\dx.
\end{aligned}
\end{equation}
Now, to estimate the term \(\int_{B_R}\varphi_n \nabla [x \cdot \nabla u] \cdot \nabla (\Delta u)\dx\), notice that we can write
\[
\begin{aligned}
    \mathrm{div}\left(\varphi_n \nabla [x\cdot \nabla u]\Delta u\right) &= \nabla\varphi_n \cdot \nabla [x\cdot \nabla u]\Delta u + \varphi_n \nabla \left(\nabla [x\cdot \nabla u]\right)\Delta u 
    \\
    &\quad+ \varphi_n \nabla [x \cdot \nabla u] \cdot \nabla (\Delta u).
\end{aligned}
\]
Since
\[
\begin{aligned}
    \varphi_n \nabla \left(\nabla [x\cdot \nabla u]\right)\Delta u  &= \varphi_n \Delta [x\cdot \nabla u] \Delta u = \varphi_n\left(2 \Delta u + x\cdot \nabla(\Delta u)\right)\Delta u 
    \\
    &=2\varphi_n (\Delta u)^2 + \varphi_n\left(x\cdot\nabla(\Delta u)\right)\Delta u,  
\end{aligned}
\]
one concludes
\[
\begin{aligned}
\varphi_n \nabla [x \cdot \nabla u] \cdot \nabla (\Delta u) &= \mathrm{div}\left(\varphi_n \nabla [x\cdot \nabla u]\Delta u\right) - \nabla\varphi_n \cdot \nabla [x\cdot \nabla u]\Delta u 
\\
&\quad- 2\varphi_n (\Delta u)^2 - \varphi_n\left(x\cdot\nabla(\Delta u)\right)\Delta u.    
\end{aligned}
\]
Combining this equality with
\[
\mathrm{div}\left(\varphi_n \frac{(\Delta u)^2}{2}x\right) = \frac{1}{2}[\nabla \varphi_n\cdot x](\Delta u)^2 + \varphi_n\left(x\cdot\nabla(\Delta u)\right)\Delta u + 2\varphi_n (\Delta u)^2
\]
we obtain
\[
\begin{aligned}
\varphi_n \nabla [x \cdot \nabla u] \cdot \nabla (\Delta u) &= \mathrm{div}\left(\varphi_n \nabla [x\cdot \nabla u]\Delta u\right) - \nabla\varphi_n \cdot \nabla [x\cdot \nabla u]\Delta u  
\\
&\quad + \frac{1}{2}[\nabla \varphi_n\cdot x](\Delta u)^2 - \mathrm{div}\left(\varphi_n \frac{(\Delta u)^2}{2}x\right).
\end{aligned}
\]
Applying \eqref{divergente-zero},
\[
\int_{B_R}\varphi_n \nabla [x \cdot \nabla u] \cdot \nabla (\Delta u)\dx = - \int_{B_R}\nabla\varphi_n \cdot \nabla [x\cdot \nabla u]\Delta u  \dx + \frac{1}{2}\int_{B_R}[\nabla \varphi_n\cdot x](\Delta u)^2\dx.
\]
Replacing this integral in \eqref{conta-div-final} it follows that
\[
\begin{aligned}
    0&=\int_{\mathbb{R}^4}[x \cdot \nabla u](\nabla \varphi_n \cdot \nabla (\Delta u))\dx - \int_{\mathbb{R}^4}\nabla\varphi_n \cdot \nabla [x\cdot \nabla u]\Delta u  \dx 
    \\
    &\quad+ \frac{1}{2}\int_{\mathbb{R}^4}[\nabla \varphi_n\cdot x](\Delta u)^2\dx
    + \frac{1}{2}\int_{\mathbb{R}^4}|x|^au^2[x\cdot\nabla \varphi_n]\dx 
     + \frac{(a+4)}{2}\int_{\mathbb{R}^4}\varphi_n|x|^a u^2\dx
     \\
     &\quad -\lambda\int_{\mathbb{R}^4}|x|^bF(u)[x\cdot\nabla \varphi_n]\dx -\lambda(b+4)\int_{\mathbb{R}^4}\varphi_n|x|^b F(u)\dx.
\end{aligned}
\]
Since \(\varphi_n \to 1\) and \(\nabla\varphi_n \to 0\), as \(n\to \infty\), applying the Lebesgue’s Dominated Convergence Theorem, we have that 
\[
0 =  \frac{a+4}{2}\int_{\mathbb{R}^4}|x|^a u^2\dx-\lambda(b+4)\int_{\mathbb{R}^4}|x|^b F(u)\dx,
\]
and hence \eqref{identidade-pohozaev} holds.

As a consequence of \eqref{identidade-pohozaev} the item \((iv)\) follows and the proof of Theorem \ref{T1} is finished.

Finally, assume that \(\lambda\leq 0\) and let
\(u\in E_{\mathrm{rad}}^{a}\) be a weak solution of \((P_{\lambda})\). Since \(a,b>-4\), we have \(a+4>0\) and \(b+4>0\). Moreover, it follows from
\((f_{2})\) that
$
F(s)>0
$ for every $s\neq 0$. By the
Pohozaev identity,
\[
0\leq \frac{a+4}{2}
\int_{\mathbb{R}^{4}}|x|^{a}u^{2}\dx
=
\lambda(b+4)
\int_{\mathbb{R}^{4}}|x|^{b}F(u)\dx\leq 0,
\] 
which implies that
\[
\int_{\mathbb{R}^{4}}|x|^{a}u^{2}\dx=0.
\]
Since \(|x|^{a}>0\) almost everywhere in \(\mathbb{R}^{4}\), it follows that
\(u=0\) almost everywhere. Hence, Problem \eqref{P} admits no
nontrivial weak solution when \(\lambda\leq 0\).

\vspace{0,3cm}

\noindent \textbf{Acknowledgments} The authors gratefully acknowledge that this work was carried out during Lorena Maria Silva's postdoctoral fellowship at the Federal University of Pernambuco (UFPE), Brazil.

\vspace{0,3cm}

\noindent \textbf{Data Availability} The authors declare that this paper has no associated data.

\vspace{0,3cm}

\noindent \textbf{Conflicts of Interest} The authors declare that there are no conflicts of interest regarding the publication of this paper.

\end{document}